\documentclass[final]{siamltex}
\usepackage[leqno]{amsmath}
\usepackage{amsbsy}
\usepackage{amssymb}
\usepackage[dvips]{graphicx}

\title{A Generalized ANOVA Dimensional Decomposition for Dependent Probability Measures \thanks{This work was supported by the U.S. National Science Foundation under Grant Numbers CMMI-0969044 and CMMI-1130147.}}

\author{Sharif Rahman\thanks{Applied Mathematics \& Computational Sciences, The University of Iowa, Iowa City, IA 52242 ({\tt sharif-rahman@uiowa.edu}).}}

\begin{document}

\maketitle

\begin{abstract}
This article explores the generalized analysis-of-variance or ANOVA
dimensional decomposition (ADD) for multivariate functions of dependent
random variables. Two notable properties, stemming from weakened annihilating
conditions, reveal that the component functions of the generalized
ADD have \emph{zero} means and are hierarchically orthogonal. By exploiting
these properties, a simple, alternative approach is presented to derive
a coupled system of equations that the generalized ADD component functions
satisfy. The coupled equations, which subsume as a special case the
classical ADD, reproduce the component functions for independent
probability measures. To determine the component functions of the
generalized ADD, a new constructive method is proposed by employing
measure-consistent, multivariate orthogonal polynomials as bases and
calculating the expansion coefficients involved from the solution
of linear algebraic equations. New generalized formulae are presented
for the second-moment characteristics, including triplets of global
sensitivity indices, for dependent probability distributions. Furthermore,
the generalized ADD leads to extended definitions of effective dimensions,
reported in the current literature for the classical ADD. Numerical
results demonstrate that the correlation structure of random variables
can significantly alter the composition of component functions, producing
widely varying global sensitivity indices and, therefore, distinct
rankings of random variables. An application to random eigenvalue analysis demonstrates
the usefulness of the proposed approximation.
\end{abstract}

\begin{keywords}
ADD, ANOVA, dimensional decomposition, multivariate orthogonal polynomials,
uncertainty quantification, global sensitivity analysis
\end{keywords}

\begin{AMS}
26B49, 41A61, 49K30, 60H35, 65C60
\end{AMS}

\pagestyle{myheadings}
\thispagestyle{plain}
\markboth{S. RAHMAN}{GENERALIZED ANOVA DIMENSIONAL DECOMPOSITION}

\section{Introduction}

Uncertainty quantification of complex systems, whether natural or engineered, is a crucial ingredient in numerous fields of engineering, science, and medicine. The remarkable growth of computing power, complemented by matching gains in algorithmic speed and accuracy, has led to near-ubiquity of computational methods for estimating the statistical moments, probability laws, and other relevant properties of such systems. However, most existing methods \cite{deb01,holtz08,wiener38}, while successful in tackling small to moderate numbers of random variables, begin to break down for truly high-dimensional problems. The root deterrence to practical computability is often related to high dimension of the multivariate integration or interpolation problem, known as the curse of dimensionality \cite{bellman57}. The dimensional decomposition of a multivariate function \cite{hoeffding48,sobol69,rahman12,kuo10} addresses the curse of dimensionality to some extent by developing an input-output behavior of complex systems with low effective dimensions \cite{caflisch97}, wherein the degrees of interactions between input variables attenuate rapidly or vanish altogether.

A well-known prototype of dimensional decomposition is the analysis-of-variance
or ANOVA dimensional decomposition (ADD), first presented by Hoeffding
in the 1940s in relation to his seminal work on $U$-statistics \cite{hoeffding48}.
Since then, ADD has been studied by numerous researchers in disparate
fields of mathematics \cite{sobol03,hickernell95}, statistics \cite{owen03,efron81},
finance \cite{griebel10}, and basic and applied sciences \cite{rabitz99},
including engineering disciplines, where its polynomial versions have
been successfully applied for uncertainty quantification of high-dimensional
complex systems \cite{rahman08,rahman09,yadav13}. However, the existing
ADD, referred to as the classical ADD in this paper, is strictly
valid for independent, product-type probability measures of random
input. In reality, there may exist significant correlation or dependence
among input variables. The author rules out the Rosenblatt transformation
\cite{rosenblatt52} or others commonly used for mapping dependent
to independent variables, as they may induce overly large nonlinearity
to a stochastic response, potentially degrading the convergence properties
of probabilistic solutions \cite{rahman09}. Therefore, the classical
ADD must be generalized for an arbitrary, non-product-type probability
measure. Doing so will require modifying the original annihilating
conditions that will endow desirable orthogonal properties, insofar as is
possible, to the generalization. Indeed, inspired by Stone \cite{stone94} and
employing a set of weakened
annihilating conditions, Hooker \cite{hooker07} provided an existential
proof of a unique ANOVA decomposition for dependent variables, referred
to as the generalized ADD in this paper, subject to a mild restriction
on the probability measure. Furthermore, he introduced a computational
method for determining the component functions of the generalized
ADD by minimizing a mean-squared error with weakened annihilating
conditions as constraints. However, the method turns out to be computationally
demanding and even potentially prohibitive when there exist a moderate
number of variables. Chastaing \emph{et al}. \cite{chastaing12} presented a boundedness
assumption on the joint probability density function of random variables for the availability
of a generalized ANOVA decomposition. Li and Rabitz \cite{li12} proposed combining
hierarchically selected univariate orthogonal polynomials and regression, applied to input-output data,
to approximate the component functions of the generalized ADD.

The purpose of this paper is threefold. Firstly, a brief exposition
of the classical ADD is given in Section 3, setting the stage for
the generalized ADD presented in Section 4. Two propositions and a
theorem, proven herein, reveal two special properties of the generalized
ADD, leading to a coupled system of equations satisfied by the component
functions. These theoretical results, which subsume the classical
ADD as a special case, are shown to reproduce the component functions
for independent probability measures. Secondly, Section 5 introduces
general multivariate orthogonal polynomials that are consistent with
the probability measures of dependent input variables. A theorem and
its proof presented in this section describe a new constructive method
for finding the component functions of the generalized ADD in terms
of measure-consistent, multivariate orthogonal polynomials. Thirdly,
the second-moment analysis of the generalized ADD is described in
Section 6. It entails global sensitivity analysis, including triplets
of sensitivity indices, for dependent probability distributions. Using
insights from the generalized ADD, extended definitions of two effective
dimensions are proposed. Numerical results, including approximate solutions of a random
eigenvalue problem, are reported in Section 7 to affirm the theoretical
findings. Mathematical notations and conclusions are defined or drawn
in Sections 2 and 8, respectively.


\section{Notation}

Let $\mathbb{N}$, $\mathbb{N}_{0}$, $\mathbb{R}$, and $\mathbb{R}_{0}^{+}$
represent the sets of positive integer (natural), non-negative integer,
real, and non-negative real numbers, respectively. For $k\in\mathbb{N}$,
denote by $\mathbb{R}^{k}$ the $k$-dimensional Euclidean space and
by $\mathbb{R}^{k\times k}$ the set of $k\times k$ real-valued matrices.
These standard notations will be used throughout the paper.

Let $(\Omega,\mathcal{F},P)$ be a complete probability space, where
$\Omega$ is a sample space, $\mathcal{F}$ is a $\sigma$-field on
$\Omega$, and $P:\mathcal{F}\to[0,1]$ is a probability measure.
With $\mathcal{B}^{N}$ representing the Borel $\sigma$-field on
$\mathbb{R}^{N}$, $N\in\mathbb{N}$, consider an $\mathbb{R}^{N}$-valued
random vector $\mathbf{X}:=(X_{1},\cdots,X_{N}):(\Omega,\mathcal{F})\to(\mathbb{R}^{N},\mathcal{B}^{N})$,
describing the statistical uncertainties in all system and input parameters
of a high-dimensional stochastic problem. The probability law of $\mathbf{X}$, assumed to be
continuous, is completely defined by its joint probability density function $f_{\mathbf{X}}:\mathbb{R}^{N}\to\mathbb{R}_{0}^{+}$.
Let $u$ be a subset of $\{1,\cdots,N\}$ with the complementary set
$-u:=\{1,\cdots,N\}\backslash u$ and cardinality $0\le|u|\le N$,
and let $\mathbf{X}_{u}=(X_{i_{1}},\cdots,X_{i_{|u|}})$, $u\neq\emptyset$,
$1\leq i_{1}<\cdots<i_{|u|}\leq N$, be a subvector of $\mathbf{X}$
with $\mathbf{X}_{-u}:=\mathbf{X}_{\{1,\cdots,N\}\backslash u}$ defining
its complementary subvector. Then, for a given $\emptyset\neq u\subseteq\{1,\cdots,N\}$,
the marginal density function of $\mathbf{X}_{u}$ is $f_{u}(\mathbf{x}_{u}):=\int_{\mathbb{R}^{N-|u|}}f_{\mathbf{X}}(\mathbf{x})d\mathbf{x}_{-u}$.

Let $y(\mathbf{X}):=y(X_{1},\cdots,X_{N}$), a real-valued, measurable
transformation on $(\Omega,\mathcal{F})$, define a high-dimensional
stochastic response of interest and $\mathcal{L}_{2}(\Omega,\mathcal{F},P)$
represent a Hilbert space of square-integrable functions $y$ with
respect to the induced generic measure $f_{\mathbf{X}}(\mathbf{x})d\mathbf{x}$
supported on $\mathbb{R}^{N}$. Although it is well known in the current
literature, Section 3 briefly describes the classical ADD, so that
it can be contrasted with the generalized ADD presented in Section
4, the main theme of this paper.

\section{Classical ANOVA Dimensional Decomposition}

The classical ADD is established by assuming independent coordinates
of $\mathbf{X}$ and selecting a product-type probability measure,
$f_{\mathbf{X}}(\mathbf{x})d\mathbf{x}=\Pi_{i=1}^{N}f_{\{i\}}(x_{i})dx_{i}$,
of $\mathbf{X}$, where $f_{\{i\}}:\mathbb{R}\to\mathbb{R}_{0}^{+}$
is the marginal probability density function of $X_{i}$, defined
on the probability triple $(\Omega_{i},\mathcal{F}_{i},P_{i})$ with
a bounded or an unbounded support on $\mathbb{R}$. The ADD, expressed
by the compact form \cite{sobol03,rahman12,kuo10}
\begin{equation}
y(\mathbf{X})={\displaystyle \sum_{u\subseteq\{1,\cdots,N\}}y_{u,C}(\mathbf{X}_{u})},\label{1}
\end{equation}
is a finite, hierarchical expansion in terms of its input variables
with increasing dimensions, where $y_{u,C}$ is a $|u|$-variate component
function describing a constant or the interactive effect of $\mathbf{X}_{u}$
on $y$ when $|u|=0$ or $|u|>0$. The symbol $C$ in the subscript
of $y_{u,C}$ is a reminder that the component functions belong to
the classical ADD. The summation in Equation \ref{1} comprises $2^{N}$
component functions, with each function depending on a group of variables
indexed by a particular subset of $\{1,\cdots,N\}$, including the
empty set $\emptyset$. Applying strong annihilating conditions, the
component functions are endowed with desirable orthogonal properties,
explained as follows.

\subsection{Strong Annihilating Conditions}

The strong annihilating conditions relevant to the classical ADD require
all non-constant component functions $y_{u,C}$ to integrate to \emph{zero}
with respect to the marginal density of each random variable in $u$,
that is \cite{sobol93, rabitz99,rahman12,kuo10},
\begin{equation}
\int_{\mathbb{R}}y_{u,C}(\mathbf{x}_{u})f_{\{i\}}(x_{i})dx_{i}=0\;\mathrm{for}\; i\in u\neq\emptyset,\label{2}
\end{equation}
resulting in two remarkable properties, described by Propositions
\ref{prop:1} and \ref{prop:2}.
\begin{proposition}
\label{prop:1}The classical ADD component functions $y_{u,C}$, where
$\emptyset\ne u\subseteq\{1,\cdots,N\}$, have zero means, i.e.,
\[
\mathbb{E}\left[y_{u,C}(\mathbf{X}_{u})\right]=0.
\]
\end{proposition}
\begin{proposition}
\label{prop:2}Two distinct classical ADD component functions $y_{u,C}$
and $y_{v,C}$, where $\emptyset\ne u\subseteq\{1,\cdots,N\}$, $\emptyset\ne v\subseteq\{1,\cdots,N\}$,
and $u\neq v$, are orthogonal, i.e., they satisfy the property
\[
\mathbb{E}\left[y_{u,C}(\mathbf{X}_{u})y_{v,C}(\mathbf{X}_{v})\right]=0.
\]
\end{proposition}

Integrating (\ref{1}) with respect to the measure $f_{-u}(\mathbf{x}_{-u})d\mathbf{x}_{-u}=\prod_{i=1,i\notin u}^{N}{\displaystyle {\displaystyle {\textstyle f_{\{i\}}(x_{i})}}dx_{i}}$,
that is, over all variables except $\mathbf{x}_{u}$, and using (\ref{2})
yields the component functions \cite{sobol03,rahman12,kuo10}
\begin{subequations}
\begin{align}
y_{\emptyset,C} & =\int_{\mathbb{R}^{N}}y(\mathbf{x})\prod_{i=1}^{N}{\displaystyle {\displaystyle {\textstyle f_{\{i\}}(x_{i})}}dx_{i}},\label{3a}\\
y_{u,C}(\mathbf{X}_{u}) & ={\displaystyle \int_{\mathbb{R}^{N-|u|}}y(\mathbf{X}_{u},\mathbf{x}_{-u})}\prod_{i=1,i\notin u}^{N}{\displaystyle {\displaystyle {\textstyle f_{\{i\}}(x_{i})}}dx_{i}}-{\displaystyle \sum_{v\subset u}}y_{v,C}(\mathbf{X}_{v}).\label{3b}
\end{align}
\label{3}
\end{subequations}
\!\!In Equation \ref{3b}, $(\mathbf{X}_{u},\mathbf{x}_{-u})$
denotes an $N$-dimensional vector whose $i$th component is $X_{i}$
if $i\in u$ and $x_{i}$ if $i\notin u.$ When $u=\emptyset$, the
sum in (\ref{3b}) vanishes, resulting in the expression of the constant
function $y_{\emptyset,C}$ in (\ref{3a}). When $u=\{1,\cdots,N\}$,
the integration in the last line of (\ref{3b}) is on the empty set,
reproducing Identity (\ref{1}) and hence finding the last function
$y_{\{1,\cdots,N\},C}$. Indeed, all component functions of $y$ in
(\ref{1}) can be obtained by interpreting literally (\ref{3b}).

Traditionally, (\ref{1}), (\ref{3a}), and (\ref{3b}) with $X_{j}$,
$j=1,\cdots,N$, following independent, standard uniform distributions,
that is, $f_{\{i\}}=1$, are identified as the classical ANOVA decomposition
\cite{sobol03}. However, recent works reveal no fundamental requirement
for a specific probability measure of $\mathbf{X}$, provided that
the resultant integrals in (\ref{3a}) and (\ref{3b}) exist and
are finite \cite{rahman12}. This generalization is trivial as long
as $\mathbf{X}$ is endowed with a product-type probability measure.

\subsection{Second-Moment Statistics}

Applying the expectation operators on $y(\mathbf{X})$ in (\ref{1})
and $(y(\mathbf{X})-\mu)^{2}$ and recognizing Propositions \ref{prop:1}
and \ref{prop:2}, the mean of $y$ is
\begin{equation}
\mu:=\mathbb{E}\left[y(\mathbf{X})\right]=y_{\emptyset,C},\label{5}
\end{equation}
whereas its variance
\begin{equation}
\sigma^2:=\mathbb{E}\left[\left(y(\mathbf{X})-\mu\right)^{2}\right]=\sum_{\emptyset\ne u\subseteq\{1,\cdots,N\}}\mathbb{E}\left[y_{u,C}^{2}(\mathbf{X}_{u})\right]\label{6}
\end{equation}
splits into variances of all \emph{zero-}mean, non-constant component
functions of $y$. According to (\ref{6}), the variance decomposition
follows the same structure of $y-y_{\emptyset,C}$ from (\ref{1}),
explaining why the acronym\textquoteleft{}\textquoteleft{}ANOVA\textquoteright{}\textquoteright{}is
also coined for the function decomposition.

\section{Generalized ANOVA Dimensional Decomposition}

Consider a dependent random vector with an arbitrary non-product type
probability density function $f_{\mathbf{X}}:\mathbb{R}^{N}\to\mathbb{R}_{0}^{+}$
that has marginal probability density function $f_{u}$ of $\mathbf{X}_{u}$,
where $\emptyset\ne u\subseteq\{1,\cdots,N\}$. Assume that the support
of $f_{\mathbf{X}}$ is grid-closed \cite{hooker07}. The grid closure
implies that there exists a grid for every point $\mathbf{x}$ of
$\mathrm{supp}(f_{\mathbf{X}})\subseteq\mathbb{R}^{N}$, that is,
for any point $\mathbf{x}\in\mathrm{supp}(f_{\mathbf{X}})$, one can
traverse in each coordinate direction and find another point $\mathbf{x}'\in\mathrm{supp}(f_{\mathbf{X}})$.
Under this mild regularity requirement, fulfilled by common probability
distributions, a square-integrable multivariate function $y$ with
respect to the marginal probability measure $f_{u}(\mathbf{x}_{u})d\mathbf{x}_{u}$
supported on $\mathbb{R}^{|u|}$ also admits a unique, finite, hierarchical
expansion \cite{hooker07}
\begin{equation}
y(\mathbf{X})={\displaystyle \sum_{u\subseteq\{1,\cdots,N\}}y_{u,G}(\mathbf{X}_{u})},\label{7}
\end{equation}
referred to as the generalized ADD, in terms of component functions
$y_{u,G}$, $u\subseteq\{1,\cdots,N\}$, of input variables with increasing
dimensions. The existence and uniqueness of the decomposition in (\ref{7}) have been proven under conditions (C.1) and (C.2) \cite{chastaing12}, but (\ref{7}) works well in practice under more general assumptions \cite{hooker07,stone94}. The symbol $G$ in the subscript of $y_{u,G}$ is meant
to distinguish the component functions of the generalized ADD from
those of the classical ADD. Similar to the classical ADD, the summation
in (\ref{7}) comprises $2^{N}$ component functions, with each function
depending on a group of variables indexed by a particular subset of
$\{1,\cdots,N\}$, including the empty set $\emptyset$. However,
the component functions of the generalized ADD, different than those
of the classical ADD, cannot be derived from the strong annihilating
conditions expressed by (\ref{2}). This is because some of the orthogonal
properties that stem from (\ref{2}) cannot be duplicated when the
random variables are dependent. Having said so, the functions $y_{u,G}$,
$u\subseteq\{1,\cdots,N\}$, can also be obtained from a similar perspective
by adjusting classical annihilating conditions, described as follows.

\subsection{Weak Annihilating Conditions}

The weak annihilating conditions appropriate for the generalized ADD
mandate all non-constant component functions $y_{u,G}$ to integrate
to \emph{zero} with respect to the marginal density of $\mathbf{X}_{u}$
in each coordinate direction of $u$, that is \cite{hooker07},
\begin{equation}
\int_{\mathbb{R}}y_{u,G}(\mathbf{x}_{u})f_{u}(\mathbf{x}_{u})dx_{i}=0\;\mathrm{for}\; i\in u\neq\emptyset.\label{8}
\end{equation}
Compared with (\ref{2}), (\ref{8}) represents a milder version,
but it still produces two remarkable properties of the generalized
ADD, described by Propositions \ref{prop:3} and \ref{prop:4}.
\begin{proposition}
\label{prop:3}The generalized ADD component functions $y_{u,G}$,
where $\emptyset\ne u\subseteq\{1,\cdots,N\}$, have zero means, i.e.,
\[
\mathbb{E}\left[y_{u,G}(\mathbf{X}_{u})\right]=0.
\]
\end{proposition}
\begin{proof}
For any subset $\emptyset\ne u\subseteq\{1,\cdots,N\}$, let $i\in u$.
Then
\begin{align*}
\mathbb{E}\left[y_{u,G}(\mathbf{X}_{u})\right] & :=\int_{\mathbb{R}^{N}}y_{u,G}(\mathbf{x}_{u})f_{\mathbf{X}}(\mathbf{x})d\mathbf{x}\\
 & =\int_{\mathbb{R}^{|u|}}y_{u,G}(\mathbf{x}_{u})f_{u}(\mathbf{x}_{u})d\mathbf{x}_{u}\\
 & =\int_{\mathbb{R}^{|u|-1}}\int_{\mathbb{R}}y_{u,G}(\mathbf{x}_{u})f_{u}(\mathbf{x}_{u})dx_{i}{\displaystyle \prod_{j\in u,j\neq i}dx_{j}}\\
 & =0,
\end{align*}
where the last line follows from using (\ref{8}).
\end{proof}
\begin{proposition}
\label{prop:4}Two distinct generalized ADD component functions $y_{u,G}$
and $y_{v,G}$, where $\emptyset\ne u\subseteq\{1,\cdots,N\}$, $\emptyset\ne v\subseteq\{1,\cdots,N\}$,
and $v\subset u$, are orthogonal, i.e., they satisfy the property
\[
\mathbb{E}\left[y_{u,G}(\mathbf{X}_{u})y_{v,G}(\mathbf{X}_{v})\right]=0.
\]
\end{proposition}
\begin{proof}
For any two subsets $\emptyset\ne u\subseteq\{1,\cdots,N\}$ and $\emptyset\ne v\subseteq\{1,\cdots,N\}$,
where $v\subset u$, the subset $u=v\cup(u\setminus v)$. Let $i\in(u\setminus v)\subseteq u$.
Then
\begin{align*}
\mathbb{E}\left[y_{u,G}(\mathbf{X}_{u})y_{v,G}(\mathbf{X}_{v})\right] & :=\int_{\mathbb{R}^{N}}y_{u,G}(\mathbf{x}_{u})y_{v,G}(\mathbf{x}_{v})f_{\mathbf{X}}(\mathbf{x})d\mathbf{x}\\
 & =\int_{\mathbb{R}^{|u|}}y_{u,G}(\mathbf{x}_{u})y_{v,G}(\mathbf{x}_{v})f_{u}(\mathbf{x}_{u})d\mathbf{x}_{u}\\
 & =\int_{\mathbb{R}^{|v|}}y_{v,G}(\mathbf{x}_{v})\int_{\mathbb{R}^{|u\setminus v|}}y_{u,G}(\mathbf{x}_{u})f_{u}(\mathbf{x}_{u})d\mathbf{x}_{u\setminus v}d\mathbf{x}_{v}\\
 & =\int_{\mathbb{R}^{|v|}}y_{v,G}(\mathbf{x}_{v})\int_{\mathbb{R}^{|u\setminus v|-1}}\int_{\mathbb{R}}y_{u,G}(\mathbf{x}_{u})f_{u}(\mathbf{x}_{u})dx_{i}\prod_{{\textstyle {j\in(u\setminus v)\atop {j\neq i}}}}dx_{j}d\mathbf{x}_{v}\\
 & =0,
\end{align*}
where the equality to \emph{zero} results from using (\ref{8}).
\end{proof}

It is elementary to show that (\ref{8}) shrinks to (\ref{2}) for independent random variables, that is, when $\mathbf{X}_{u}$ follows a product-type probability measure $f_{u}(\mathbf{x}_{u})=\Pi_{i \in u}f_{\{i\}}(x_{i})$ for $\emptyset\ne u\subseteq\{1,\cdots,N\}$.

From Propositions \ref{prop:1} and \ref{prop:3}, all non-constant
component functions of ADD, whether classical or generalized, have
\emph{zero} means. Therefore, a non-product type probability measure,
relevant to the generalized ADD, does not vitiate the first-moment
properties of the classical ADD. However, Propositions \ref{prop:2}
and \ref{prop:4}, which describe the second-moment properties of
ADD, tell a slightly different tale: any two distinct non-constant
component functions of the classical ADD are orthogonal, whereas two
distinct non-constant component functions of the generalized ADD are
orthogonal only if the index set of one function is a proper subset
of the index set of the other function. As an example, consider $N=3$
with $2^{3}-1=7$ non-constant component functions. Then, the generalized
ADD permits orthogonality between (1) $y_{\{i\},G}$ and $y_{\{i_{1},i_{2}\},G}$,
where $i=1,2,3$, ($i_{1}=i,\: i_{2}=1,2,3,\: i_{1}<i_{2})$, ($i_{2}=i,\: i_{1}=1,2,3,\: i_{1}<i_{2})$;
(2) $y_{\{i\},G}$ and $y_{\{123\},G}$, where $i=1,2,3$; and (3)
$y_{\{i_{1},i_{2}\},G}$ and $y_{\{1,2,3\},G}$, where $i_{1},i_{2}=1,2,3,\: i_{1}<i_{2}$.
This nested orthogonality, originally presented and referred to as the hierarchical orthogonality
by Hooker \cite{hooker07}, is the result of imposing weakened annihilating conditions on the generalized ADD.

\subsection{A Coupled System of Equations}

Hooker \cite{hooker07} proposed finding component functions of the
generalized ADD by minimizing a mean-squared error subject to the
hierarchical orthogonality described by Propositions \ref{prop:3}
and \ref{prop:4}. A simpler alternative proposed here entails integrating
(\ref{7}) with respect to a judiciously chosen marginal probability
measure and implementing the weak annihilating conditions when required.
Lemma \ref{lem:5} and Theorem \ref{thm:6} illuminate this alternative
approach, which sidesteps the need to solve the optimization problem
altogether. The end result is a coupled system of equations satisfied
by component functions.
\begin{lemma}
\label{lem:5}The generalized ADD component functions $y_{v,G}$,
$\emptyset \neq v\subseteq\{1,\cdots,N\}$, of a square-integrable function $y:\mathbb{R}^{N}\to\mathbb{R}$,
when integrated with respect to the probability measure $f_{-u}(\mathbf{x}_{-u})d\mathbf{x}_{-u}$,
$u\subseteq\{1,\cdots,N\}$, reduce to
\begin{equation}
\begin{array}{l}
{\displaystyle \int_{\mathbb{R}^{N-|u|}}y_{v,G}(\mathbf{x}_{v})f_{-u}(\mathbf{x}_{-u})d\mathbf{x}_{-u}}=\\
\qquad\begin{cases}
\int_{\mathbb{R}^{|v\cap-u|}}y_{v,G}(\mathbf{x}_{v})f_{v\cap-u}(\mathbf{x}_{v\cap-u})d\mathbf{x}_{v\cap-u}\; & \text{if}~~v\cap u\neq\emptyset\:\mathrm{and}\: v\nsubseteq u,\\
y_{v,G}(\mathbf{x}_{v}) & \text{if}~~v\cap u\neq\emptyset\:\mathrm{and}\: v\subseteq u,\\
0 & \text{if}~~v\cap u=\emptyset.
\end{cases}
\end{array}\label{9}
\end{equation}
\end{lemma}

\begin{proof}
For any two subsets $\emptyset \neq v\subseteq\{1,\cdots,N\}$, $u\subseteq\{1,\cdots,N\}$,
one can write $(v\cap-u)\subseteq-u$ and $-u=(-u\setminus(v\cap-u))\cup(v\cap-u)$.
Let $v\cap u\neq\emptyset$, where $v\nsubseteq u$ in general. Then
one of the two non-trivial results of (\ref{9}) is obtained as
\begin{equation}
\begin{split}{\displaystyle \int_{\mathbb{R}^{N-|u|}}\!\! y_{v,G}(\mathbf{x}_{v})f_{-u}(\mathbf{x}_{-u})d\mathbf{x}_{-u}} & \!=\!{\displaystyle \int_{\mathbb{R}^{|v\cap-u|}}\!\! y_{v,G}(\mathbf{x}_{v})\int_{\mathbb{R}^{N-|u|-|v\cap-u|}}\!\!\!\! f_{-u}(\mathbf{x}_{-u\setminus(v\cap-u)},\mathbf{x}_{v\cap-u})}\\
 & \;\;\;\times d\mathbf{x}_{-u\setminus(v\cap-u)}d\mathbf{x}_{v\cap-u}\\
 & \!=\!{\displaystyle \int_{\mathbb{R}^{|v\cap-u|}}y_{v,G}(\mathbf{x}_{v})f_{v\cap-u}(\mathbf{x}_{v\cap-u})d\mathbf{x}_{v\cap-u}}.
\end{split}
\label{10}
\end{equation}
If $v\subseteq u$, then $y_{v}(\mathbf{x}_{v})$ does not depend
on $\mathbf{x}_{-u}$, resulting in
\[
{\displaystyle \int_{\mathbb{R}^{N-|u|}}y_{v,G}(\mathbf{x}_{v})f_{-u}(\mathbf{x}_{-u})d\mathbf{x}_{-u}}=y_{v,G}(\mathbf{x}_{v}){\displaystyle \int_{\mathbb{R}^{N-|u|}}f_{-u}(\mathbf{x}_{-u})d\mathbf{x}_{-u}}=y_{v,G}(\mathbf{x}_{v}),
\]
the other non-trivial result of (\ref{9}). Finally, if $v\cap u=\emptyset$,
then $v\cap-u=v$. Let $i\in v$. Therefore, the last line of (\ref{10}), also valid for $v \cap u=\emptyset$, becomes
\[
\begin{split}{\displaystyle \int_{\mathbb{R}^{N-|u|}}y_{v,G}(\mathbf{x}_{v})f_{-u}(\mathbf{x}_{-u})d\mathbf{x}_{-u}} & ={\displaystyle \int_{\mathbb{R}^{|v|}}y_{v,G}(\mathbf{x}_{v})f_{v}(\mathbf{x}_{v})d\mathbf{x}_{v}}\\
 & ={\displaystyle \int_{\mathbb{R}^{|v|-1}}\left(\int_{\mathbb{R}}y_{v,G}(\mathbf{x}_{v})f_{v}(\mathbf{x}_{v})dx_{i}\right)}{\displaystyle \prod_{j\in v,j\neq i}dx_{j}}\\
 & =0,
\end{split}
\]
where the equality to \emph{zero} follows from using (\ref{8}).\end{proof}
\begin{theorem}
\label{thm:6}The generalized ADD component functions $y_{u,G}$,
$u\subseteq\{1,\cdots,N\}$, of a square-integrable function $y:\mathbb{R}^{N}\to\mathbb{R}$
for a given probability measure $f_{\mathbf{X}}(\mathbf{x})d\mathbf{x}$
of $\mathbf{X}\in\mathbb{R}^{N}$ satisfy
\begin{subequations}
\begin{align}
y_{\emptyset,G} & =\int_{\mathbb{R}^{N}}y(\mathbf{x})f_{\mathbf{X}}(\mathbf{x})d\mathbf{x},\label{11a}\\
y_{u,G}(\mathbf{X}_{u}) & =\int_{\mathbb{R}^{N-|u|}}y(\mathbf{X}_{u},\mathbf{x}_{-u})f_{-u}(\mathbf{x}_{-u})d\mathbf{x}_{-u}-{\displaystyle \sum_{v\subset u}y_{v,G}(\mathbf{X}_{v})}-\nonumber \\
 & \;\;\;{\displaystyle \sum_{{\textstyle {\emptyset\ne v\subseteq\{1,\cdots,N\}\atop v\cap u\ne\emptyset,v\nsubseteq u}}}{\displaystyle \int_{\mathbb{R}^{|v\cap-u|}}y_{v,G}(\mathbf{X}_{v})f_{v\cap-u}(\mathbf{x}_{v\cap-u})d\mathbf{x}_{v\cap-u}}}.\label{11b}
\end{align}
\end{subequations}
\end{theorem}
\begin{proof}
Changing the dummy index from $u$ to $v$, replacing $\mathbf{X}$
with $\mathbf{x}$, and integrating both sides of (\ref{7}) with
respect to the measure $f_{-u}(\mathbf{x}_{-u})d\mathbf{x}_{-u}$,
that is, over all variables except $\mathbf{x}_{u}$, yields
\begin{equation}
\int_{\mathbb{R}^{N-|u|}}y(\mathbf{x})f_{-u}(\mathbf{x}_{-u})d\mathbf{x}_{-u}={\displaystyle \sum_{v\subseteq\{1,\cdots,N\}}\int_{\mathbb{R}^{N-|u|}}y_{v,G}(\mathbf{x}_{v})f_{-u}(\mathbf{x}_{-u})d\mathbf{x}_{-u}},\label{12}
\end{equation}
which is valid for any $u\subseteq\{1,\cdots,N\}$, including the
empty set $\emptyset$. To obtain the constant component function
of $y$, let $u=\emptyset$. Then $-u=\{1,\cdots,N\}$ and $f_{-u}(\mathbf{x}_{-u})d\mathbf{x}_{-u}=f_{\mathbf{X}}(\mathbf{x})d\mathbf{x}$,
resulting in
\begin{equation}
\begin{array}{rcl}
\int_{\mathbb{R}^{N}}y(\mathbf{x})f_{\mathbf{X}}(\mathbf{x})d\mathbf{x} & = & y_{\emptyset,G}+{\displaystyle \sum_{\emptyset\neq v\subseteq\{1,\cdots,N\}}\int_{\mathbb{R}^{N}}y_{v,G}(\mathbf{x}_{v})f_{\mathbf{X}}(\mathbf{x})d\mathbf{x}}\\
 & = & y_{\emptyset,G}+{\displaystyle \sum_{\emptyset\neq v\subseteq\{1,\cdots,N\}}\mathbb{E}\left[y_{v,G}(\mathbf{X}_{v})\right]}.
\end{array}\label{13}
\end{equation}
Invoking Proposition \ref{prop:3}, each expectation of the sum in
(\ref{13}) vanishes, yielding (\ref{11a}). To derive the non-constant
component functions, apply Lemma \ref{lem:5}, that is, (\ref{9})
to simplify the right side of (\ref{12}) into
\begin{equation}
\begin{array}{l}
{\displaystyle \sum_{v\subseteq\{1,\cdots,N\}}\int_{\mathbb{R}^{N-|u|}}y_{v,G}(\mathbf{x}_{v})f_{-u}(\mathbf{x}_{-u})d\mathbf{x}_{-u}}=\\
y_{u,G}(\mathbf{x}_{u})+{\displaystyle \sum_{v\subset u}}y_{v,G}(\mathbf{x}_{v})+{\displaystyle \sum_{{\textstyle {\emptyset\ne v\subseteq\{1,\cdots,N\}\atop v\cap u\ne\emptyset,v\nsubseteq u}}}{\displaystyle \int_{\mathbb{R}^{|v\cap-u|}}y_{v,G}(\mathbf{x}_{v})f_{v\cap-u}(\mathbf{x}_{v\cap-u})d\mathbf{x}_{v\cap-u}}}
\end{array}\label{14}
\end{equation}
with $\subset$ representing the proper subset (strict inclusion). Substituting
(\ref{14}) into (\ref{12}) produces (\ref{11b}), completing the
proof.
\end{proof}

The constant component function of ADD, whether classical ($y_{\emptyset,C}$)
or generalized ($y_{\emptyset,G}$), is the same as the expected value
of $y(\mathbf{X})$. According to (\ref{3b}), all non-constant component
functions of the classical ADD are hierarchically ordered in terms
of the cardinality of subsets of $\{1,\cdots,N\}$ and are determined
sequentially. This is possible because for a given $\emptyset\neq u\subseteq\{1,\cdots,N\}$,
the component function $y_{u,C}$ depends only on the component functions
$y_{v,C}$ such that $v\subset u$, including $\emptyset$. In contrast,
the component functions of the generalized ADD, satisfying (\ref{11b}),
are coupled and must be solved simultaneously. In the latter case,
for a given $\emptyset\neq u\subseteq\{1,\cdots,N\}$, the component
function $y_{u,G}$ depends not only on the component functions $y_{v,G}$,
where $v\subset u$, but also on the component functions $y_{v,G}$,
where $v\cap u\ne\emptyset$, $v\nsubseteq u$. As an example, consider
$u=\{1\}$ and $N=3$. The classical and generalized component functions
depending on $x_{1}$ are
\[
y_{\{1\},C}=\int_{\mathbb{R}^{2}}y(x_{1},x_{2},x_{3}){\displaystyle f_{\{2\}}(x_{2})f_{\{3\}}(x_{3})dx_{2}dx_{3}}-y_{\emptyset,C}
\]
and
\[
\begin{array}{rcl}
y_{\{1\},G} & = & \int_{\mathbb{R}^{2}}y(x_{1},x_{2},x_{3}){\displaystyle f_{\{2,3\}}(x_{2},x_{3})dx_{2}dx_{3}}-y_{\emptyset,G}-\\
 &  & \int_{\mathbb{R}}y_{\{1,2\},G}(x_{1},x_{2}){\displaystyle f_{\{2\}}(x_{2})dx_{2}-\int_{\mathbb{R}}y_{\{1,3\},G}(x_{1},x_{3}){\displaystyle f_{\{3\}}(x_{3})dx_{3}-}}\\
 &  & \int_{\mathbb{R}^{2}}y_{\{1,2,3\},G}(x_{1},x_{2},x_{3}){\displaystyle f_{\{2,3\}}(x_{2},x_{3})dx_{2}dx_{3}},
\end{array}
\]
respectively. For the generalized ADD, there exist $2^{N}-1$ such
coupled equations, the right number of equations to determine uniquely
all non-constant component functions. A new computational method solving this system of equations will be formally presented in the following section.

\begin{corollary}
\label{corr:r1}The univariate, bivariate, and trivariate component functions of a square-integrable function $y:\mathbb{R}^{N} \to \mathbb{R}$, obtained by setting (1) $u=\{i\}$; $i=1,\cdots,N$; $1 \le N < \infty$, (2) $u=\{i_1,i_2\}$; $i_1=1,\cdots,N-1$; $i_2=i_1+1; \cdots,N$, $2 \le N < \infty$, and (3) $u=\{i_1,i_2,i_3\}$; $i_1=1,\cdots,N-2$; $i_2=i_1+1,\cdots,N-1$; $i_3=i_2+1,\cdots,N$; $3 \le N < \infty$, respectively, in (\ref{11b}) are
\begin{equation}
\begin{split}
y_{\{i\},G}(X_{i}) & = \int_{\mathbb{R}^{N-1}}y(X_{i},\mathbf{x}_{-\{i\}})f_{-\{i\}}(\mathbf{x}_{-\{i\}})d\mathbf{x}_{-\{i\}}-y_{\emptyset,G}-\\
& \;\;\;{\displaystyle \sum_{{\textstyle {\emptyset\ne v\subseteq\{1,\cdots,N\}\atop v\cap \{i\}\ne\emptyset,v\nsubseteq \{i\}}}}{\displaystyle \int_{\mathbb{R}^{|v\cap-\{i\}|}}y_{v,G}(\mathbf{x}_{v})f_{v\cap-\{i\}}(\mathbf{x}_{v\cap-\{i\}})d\mathbf{x}_{v\cap-\{i\}}}}, \\
\end{split}
\nonumber
\end{equation}
\begin{equation}
\begin{array}{l}
  y_{\{i_1,i_2\},G}(X_{i_1},X_{i_2}) =  \\
  \int_{\mathbb{R}^{N-2}}y(X_{i_1},X_{i_2},\mathbf{x}_{-\{i_1,i_2\}})f_{-\{i_1,i_2\}}(\mathbf{x}_{-\{i_1,i_2\}})d\mathbf{x}_{-\{i_1,i_2\}}-  \\
  y_{\emptyset,G}-y_{\{i_1\},G}(X_{i_1})-y_{\{i_2\},G}(X_{i_2})-  \\
  {\displaystyle \sum_{{\textstyle {\emptyset\ne v\subseteq\{1,\cdots,N\}\atop v\cap \{i_1,i_2\}\ne\emptyset,v\nsubseteq \{i_1,i_2\}}}}{\displaystyle \int_{\mathbb{R}^{|v\cap-\{i_1,i_2\}|}}y_{v,G}(\mathbf{x}_{v})f_{v\cap-\{i_1,i_2\}}(\mathbf{x}_{v\cap-\{i_1,i_2\}})d\mathbf{x}_{v\cap-\{i_1,i_2\}}}},
\end{array}
\nonumber
\end{equation}
\begin{equation}
\begin{array}{l}
  y_{\{i_1,i_2,i_3\},G}(X_{i_1},X_{i_2},X_{i_3}) =  \\
  \int_{\mathbb{R}^{N-3}}y(X_{i_1},X_{i_2},X_{i_3},\mathbf{x}_{-\{i_1,i_2,i_3\}})f_{-\{i_1,i_2,i_3\}}(\mathbf{x}_{-\{i_1,i_2,i_3\}})d\mathbf{x}_{-\{i_1,i_2,i_3\}}-  \\
  y_{\emptyset,G}-y_{\{i_1\},G}(X_{i_1})-y_{\{i_2\},G}(X_{i_2})-y_{\{i_3\},G}(X_{i_3})-  \\
  y_{\{i_1,i_2\},G}(X_{i_1},X_{i_2})-y_{\{i_1,i_3\},G}(X_{i_1},X_{i_3})-y_{\{i_2,i_3\},G}(X_{i_2},X_{i_3})-  \\
  {\displaystyle \sum_{{\textstyle {\emptyset\ne v\subseteq\{1,\cdots,N\}\atop v\cap \{i_1,i_2,i_3\}\ne\emptyset,v\nsubseteq \{i_1,i_2,i_3\}}}}\!\!\!\!\!\!\!\!\!\!\!\!{\displaystyle \int_{\mathbb{R}^{|v\cap-\{i_1,i_2,i_3\}|}}y_{v,G}(\mathbf{x}_{v})f_{v\cap-\{i_1,i_2,i_3\}}(\mathbf{x}_{v\cap-\{i_1,i_2,i_3\}})d\mathbf{x}_{v\cap-\{i_1,i_2,i_3\}}}}.
\end{array}
\nonumber
\end{equation}
\end{corollary}

The specialized formulae for the component functions in Corollary \ref{corr:r1} were previously derived by Li and Rabitz \cite{li12}.  Theorem \ref{thm:6}, in contrast, is general, and provides a single master formula to concisely represent all component functions of the generalized ADD.

\begin{corollary}
\label{corr:7}If $\mathbf{X}=(X_{1},\cdots,X_{N})\in\mathbb{R}^{N}$
comprises independent random variables, which follow arbitrary probability
measures $f_{\{i\}}(x_{i})dx_{i}$, $i=1,\cdots,N$, then the generalized
ADD degenerates to the classical ADD.
\end{corollary}
\begin{proof}
For independent coordinates of $\mathbf{X}$, all joint probability
density functions are products of their marginals, that is, $f_{u}(\mathbf{x}_{u})=\Pi_{i\in u}f_{\{i\}}(x_{i})$,
where $\emptyset\neq u\subseteq\{1,\cdots,N\}$. Using this product
structure of probability measures, which makes the strong and weak
annihilating conditions coincide, it is elementary to show that $y_{u,G}=y_{u,C}$
for any $u\subseteq\{1,\cdots,N\}$, including $y_{\emptyset,G}=y_{\emptyset,C}$.
Therefore, the generalized ADD reduces to the classical ADD.
\end{proof}

\section{A Constructive Method for Determining Component Functions}

This section presents a new computational method, employing measure-consistent,
multivariate orthonormal polynomials as basis functions, for solving
the coupled system of equations satisfied by the component functions
of the generalized ADD.

\subsection{Multivariate Orthonormal Polynomials}

For the rest of the paper, the standard multi-index notation will
be used in describing orthogonal polynomials in several variables.
Accordingly, for a given $\emptyset\neq u\subseteq\{1,\cdots,N\}$,
$1\le|u|\le N$, let $\mathbf{j}_{|u|}=(j_{1},\cdots,j_{|u|})\in\mathbb{N}_{0}^{|u|}$
represent a $|u|$-dimensional multi-index with each component a non-negative
integer. For $\mathbf{j}_{|u|}\in\mathbb{N}_{0}^{|u|}$ and $\mathbf{x}_{u}=(x_{i_{1}},\cdots,x_{i_{|u|}})\in\mathbb{R}^{|u|}$,
where $1\leq i_{1}<\cdots<i_{|u|}\leq N$, a monomial in $\mathbf{x}_{u}$
of index $\mathbf{j}_{|u|}$ is defined by $\mathbf{x}_{u}^{\mathbf{j}_{|u|}}:=x_{i_{1}}^{j_{1}} \times \cdots \times x_{i_{|u|}}^{j_{|u|}}$.
The non-negative integer $|\mathbf{j}_{|u|}|:=j_{1}+\cdots+j_{|u|}$,
which is equal to the 1-norm of $\mathbf{j}_{|u|}$, is called the
total degree of $\mathbf{x}_{u}^{\mathbf{j}_{|u|}}$. A linear combination
of $\mathbf{x}_{u}^{\mathbf{j}_{|u|}}$, where $|\mathbf{j}_{|u|}|=m_{u}$
and $m_{u}\in\mathbb{N}$, is a homogeneous polynomial of degree $m_{u}$.
Denote by $\mathcal{P}_{m_{u}}^{u}:=\mathrm{span}\{\mathbf{x}_{u}^{\mathbf{j}_{|u|}}:|\mathbf{j}_{|u|}|=m_{u},\,\mathbf{j}_{|u|}\in\mathbb{N}_{0}^{|u|}\}$
the space of homogeneous polynomials of degree $m_{u}$, by $\Pi_{m_{u}}^{u}:=\mathrm{span}\{\mathbf{x}_{u}^{\mathbf{j}_{|u|}}:|\mathbf{j}_{|u|}|\le m_{u},\,\mathbf{j}_{|u|}\in\mathbb{N}_{0}^{|u|}\}$
the space of polynomials of degree at most $m_{u}$, and by $\Pi^{u}$
the space of all polynomials of $|u|$ variables. It is well known
that \cite{dunkl01}
\[
\dim\mathcal{P}_{m_{u}}^{u}=\binom{m_{u}+|u|-1}{m_{u}}\;\mathrm{and}\;\dim\Pi_{m_{u}}^{u}=\binom{m_{u}+|u|}{m_{u}}.
\]
Assume that, for $\mathbf{j}_{|u|}\in\mathbb{N}_{0}^{|u|}$, the moments
$\int_{\mathbb{R}^{|u|}}\mathbf{x}_{u}^{\mathbf{j}_{|u|}}f_{u}(\mathbf{x}_{u})d\mathbf{x}_{u}$
of $\mathbf{X}_{u}$ exist and are finite, and $\int_{\mathbb{R}^{|u|}}\psi_{u\mathbf{j}_{|u|}}^{2}(\mathbf{x}_{u})f_{u}(\mathbf{x}_{u})d\mathbf{x}_{u}>0$
for every $\psi_{u\mathbf{j}_{|u|}}\in\Pi^{u}$, where $\psi_{u\mathbf{j}_{|u|}}\neq0$
is a polynomial in $\mathbf{x}_{u}$ of degree $\mathbf{j}_{|u|}$.
Consistent with the probability measure $f_{u}(\mathbf{x}_{u})d\mathbf{x}_{u}$,
define an inner product
\begin{equation}
(g,h)_{f_{u}}:=\int_{\mathbb{R}^{|u|}}g(\mathbf{x}_{u})h(\mathbf{x}_{u})f_{u}(\mathbf{x}_{u})d\mathbf{x}_{u}=:\mathbb{E}\left[g(\mathbf{X}_{u})h(\mathbf{X}_{u})\right]\label{15}
\end{equation}
of two $|u|$-variate functions $g$ and $h$. Then there exist orthogonal
polynomials in $\mathbf{x}_{u}$ with respect to the inner product
defined by (\ref{15}). More precisely, a polynomial $\psi_{u\mathbf{j}_{|u|}}\in\Pi_{m_{u}}^{u}$
is called orthogonal with respect to $(\cdot,\cdot)_{f_{u}}$ if $(\psi_{u\mathbf{j}_{|u|}},\psi_{u\mathbf{k}_{|u|}})_{f_{u}}=\mathbb{E}\left[\psi_{u\mathbf{j}_{|u|}}(\mathbf{X}_{u})\psi_{u\mathbf{k}_{|u|}}(\mathbf{X}_{u})\right]=0$
for $|\mathbf{k}_{|u|}|<|\mathbf{j}_{|u|}|$, that is for all $\psi_{u\mathbf{k}_{|u|}}\in\Pi_{m_{u}-1}^{u}$.
This means that $\psi_{u\mathbf{j}_{|u|}}$ is orthogonal to all polynomials
of lower degrees, but it may not be orthogonal to other orthogonal
polynomials of the same degree. Define $\mathcal{V}_{m_{u}}^{u}:=\{\psi_{u\mathbf{j}_{|u|}}\in\Pi_{m_{u}}^{u}:(\psi_{u\mathbf{j}_{|u|}},\psi_{u\mathbf{k}_{|u|}})_{f_{u}}=0,\:\psi_{u\mathbf{k}_{|u|}}\in\Pi_{m_{u}-1}^{u}\}$
as the space of orthogonal polynomials of degree of exactly $m_{u}$.
It is elementary to show that the $\dim\mathcal{V}_{m_{u}}^{u}=\dim\mathcal{P}_{m_{u}}^{u}$.
If, in addition, $(\psi_{u\mathbf{j}_{|u|}},\psi_{u\mathbf{j}_{|u|}})_{f_{u}}=\mathbb{E}\left[\psi_{u\mathbf{j}_{|u|}}^{2}(\mathbf{X}_{u})\right]=1$,
then $\psi_{u\mathbf{j}_{|u|}}$ is called an orthonormal polynomial
in $\mathbf{x}_{u}$ of degree $|\mathbf{j}_{|u|}|$, to be used in
the remainder of this paper.

\subsection{Fourier-Polynomial Expansions}

Let $\{\psi_{u\mathbf{j}_{|u|}}(\mathbf{X}_{u}),\,\mathbf{j}_{|u|}\in\mathbb{N}_{0}^{|u|}\}$
be a set of multivariate orthonormal polynomials that is consistent
with the probability measure $f_{u}(\mathbf{x}_{u})d\mathbf{x}_{u}$
of $\mathbf{X}_{u}$. For $\emptyset\ne u=\{i_{1},\cdots,i_{|u|}\}\subseteq\{1,\cdots,N\}$,
where $1\le|u|\le N$ and $1\le i_{1}<\cdots<i_{|u|}\le N$, let $(\Omega_{u},\mathcal{F}_{u},P_{u})$
be the probability triple of $\mathbf{X}_{u}=(X_{i_{1}},\cdots,X_{i_{|u|}})$.
Denote the associated space of the $|u|$-variate component functions
of $y$ by
\[
\mathcal{L}_{2}(\Omega_{u},\mathcal{F}_{u},P_{u}):=
\left\{y_{u,G}:\int_{\mathbb{R}^{|u|}}y_{u,G}^{2}(\mathbf{x}_{u})f_{u}(\mathbf{x}_{u})d\mathbf{x}_{u}<\infty
\right\},
\]
which is a Hilbert space. Then $\{\psi_{u\mathbf{j}_{|u|}}(\mathbf{X}_{u}),\,\mathbf{j}_{|u|}\in\mathbb{N}_{0}^{|u|},j_{1},\cdots,j_{|u|} \neq 0\}$,
if it is dense, constitutes a basis of $\mathcal{L}_{2}(\Omega_{u},\mathcal{F}_{u},P_{u})$.
The standard Hilbert space theory states that every non-constant component function $y_{u,G}\in\mathcal{L}_{2}(\Omega_{u},\mathcal{F}_{u},P_{u})$
of $y$ can be expanded as \cite{dunkl01}
\begin{equation}
y_{u,G}(\mathbf{X}_{u})={\displaystyle \sum_{{\textstyle {\mathbf{j}_{|u|}\in\mathbb{N}_{0}^{|u|}\atop j_{1},\cdots,j_{|u|}\neq0}}}}C_{u\mathbf{j}_{|u|}}\psi_{u\mathbf{j}_{|u|}}(\mathbf{X}_{u})\label{16}
\end{equation}
with
\begin{equation}
C_{u\mathbf{j}_{|u|}}:=\int_{\mathbb{R}^{|u|}}y_{u,G}(\mathbf{x}_{u})\psi_{u\mathbf{j}_{|u|}}(\mathbf{x}_{u})f_{u}(\mathbf{x}_{u})d\mathbf{x}_{u}
\label{17}
\end{equation}
defining associated expansion coefficients. Note that the summation
in (\ref{16}) precludes $j_{1},\cdots,j_{|u|}=0$, that is, the individual
degree of each variable $X_{i}$ in $\psi_{u\mathbf{j}_{|u|}}$, $i\in u$, cannot be \emph{zero} since $y_{u,G}$ is a strictly $|u|$-variate function and has a \emph{zero} mean following Proposition \ref{prop:3}. For a more precise interpretation, the selection of multivariate Hermite polynomials as basis functions is described as follows.

Consider quadratic approximations of the univariate ($|u|=1$) and bivariate ($|u|=2$) component functions of $y(\mathbf{X})$, where $\mathbf{X}=(X_{1},\cdots,X_{N})$ is a \emph{zero}-mean, $N$-dimensional Gaussian random vector with positive-definite covariance matrix $\mathbf{\Sigma}_{\mathbf{X}}=\mathbb{E}[\mathbf{X}\mathbf{X}{}^{T}]=[\rho_{ij}\sigma_i\sigma_j]$, comprising variances $\sigma_{i}^{2}=1$ of $X_{i}$
and correlation coefficients $\rho_{ij}$ between $X_{i}$ and $X_{j}$,
$i,j=1,\cdots,N$, and joint probability density function
\begin{equation}
f_{\mathbf{X}}(\mathbf{x})=\left(2\pi\right)^{-\frac{N}{2}}\left(\det\mathbf{\Sigma}_{\mathbf{X}}\right)^{-\frac{1}{2}}\exp\left[-\frac{1}{2}\mathbf{x}{}^{T}\mathbf{\Sigma}_{\mathbf{X}}^{-1}\mathbf{x}\right]=:\phi_{\mathbf{X}}(\mathbf{x};\mathbf{\Sigma}_{\mathbf{X}}).
\label{r9}
\end{equation}
The marginal probability densities of $\mathbf{X}_{u}$, $\emptyset\ne u\subseteq\{1,\cdots,N\}$, are also Gaussian, and are easily derived as
\begin{equation}
f_{u}(\mathbf{x}_{u})=\left(2\pi\right)^{-\frac{|u|}{2}}\left(\det\mathbf{\Sigma}_{u}\right)^{-\frac{1}{2}}\exp\left[-\frac{1}{2}\mathbf{x}_{u}{}^{T}\mathbf{\Sigma}_{u}^{-1}\mathbf{x}_{u}\right]=:\phi_{u}(\mathbf{x}_{u};\mathbf{\Sigma}_{u}),
\label{r10}
\end{equation}
where $\mathbf{\Sigma}_{u}:=\mathbb{E}[\mathbf{X}_{u}\mathbf{X}_{u}{}^{T}]$
is the covariance matrix of $\mathbf{X}_{u}$. The probability density
$\phi_{u}(\mathbf{x}_{u};\mathbf{\Sigma}_{u})$ induces multivariate Hermite orthogonal polynomials
\begin{equation}
\tilde{\psi}_{u\mathbf{j}_{|u|}}(\mathbf{x}_{u})=
\frac{(-1)^{|\mathbf{j}_{|u|}|}}{\phi_{u}(\mathbf{x}_{u};\mathbf{\Sigma}_{u})}\left(\frac{\partial}{\partial\mathbf{x}_{u}}\right)^{\mathbf{j}_{|u|}}\phi_{u}(\mathbf{x}_{u};\mathbf{\Sigma}_{u}),\;\mathbf{j}_{|u|}\in\mathbb{N}_{0}^{|u|},
\label{r4}
\end{equation}
where ${\displaystyle (\partial/\partial\mathbf{x}_{u})^{\mathbf{j}_{|u|}}:=\partial^{j_{1}+\cdots+j_{|u|}}/\partial x_{i_{1}}^{j_{1}}\cdots\partial x_{i_{|u|}}^{j_{|u|}}}$.
They eventually form a set of multivariate Hermite orthonormal polynomials
\begin{equation}
\left\{\psi_{u\mathbf{j}_{|u|}}:=\tilde{\psi}_{u\mathbf{j}_{|u|}}/(\tilde{\psi}_{u\mathbf{j}_{|u|}},\tilde{\psi}_{u\mathbf{j}_{|u|}})_{\phi_{u}},\,\mathbf{j}_{|u|}\in\mathbb{N}_{0}^{|u|}\right\}
\label{r5}
\end{equation}
that are consistent with the probability measure $\phi_{u}(\mathbf{x}_{u};\mathbf{\Sigma}_{u})d\mathbf{x}_{u}$ of $\mathbf{X}_{u}$. For example, when $u=\{i\}$, $i=1,\cdots,N$, and $j_1\le 2$, (\ref{r4}) and (\ref{r5}) yield
\[
\psi_{\{i\}0}(X_i)=1,~\psi_{\{i\}1}(X_i)=X_i,~\psi_{\{i\}2}(X_i)=\frac{X_i^2-1}{\sqrt{2}},
\]
the sequence of orthonormal polynomials for quadratic approximation of any square-integrable univariate function of $X_i$. Clearly, the complete basis set for a general function is $\{\psi_{\{i\}0},\psi_{\{i\}1},\psi_{\{i\}2}\}$.
However, since $\mathbb{E}[y_{\{i\},G}(X_i)]=0$ as per Proposition \ref{prop:3}, only a linear combination of $\psi_{\{i\}1}$ and $\psi_{\{i\}2}$, without including $\psi_{\{i\}0}$, that is, the
basis subset $\{\psi_{\{i\}1}, \psi_{\{i\}2}\}$
is sufficient to approximate $y_{\{i\},G}$. Similarly, when $u=\{i_1,i_2\}$, $i_1=1,\cdots,N-1$, $i_2=i_1+1,\cdots,N$, and $|\mathbf{j}_{2}|\le 2$, (\ref{r4}) and (\ref{r5}) result in
\[
\begin{array}{c}
\psi_{\{i_1,i_2\}00}(X_{i_1},X_{i_2})=1,\\
\psi_{\{i_1,i_2\}10}(X_{i_1},X_{i_2})=\dfrac{X_{i_1}-\rho_{i_1 i_2}X_{i_2}}{\sqrt{1-\rho_{i_1 i_2}^2}},~
\psi_{\{i_1,i_2\}01}(X_{i_1},X_{i_2})=\dfrac{X_{i_2}-\rho_{i_1 i_2}X_{i_1}}{\sqrt{1-\rho_{i_1 i_2}^2}},\\
\psi_{\{i_1,i_2\}20}(X_{i_1},X_{i_2})=\dfrac{X_{i_1}^2+\rho_{i_1 i_2}^2 \left( 1+X_{i_2}^2 \right) -
2 \rho_{i_1 i_2} X_{i_1} X_{i_2} - 1}{\sqrt{2}\left( 1 - \rho_{i_1 i_2}^2 \right)},\\
\psi_{\{i_1,i_2\}02}(X_{i_1},X_{i_2})=\dfrac{X_{i_2}^2+\rho_{i_1 i_2}^2 \left( 1+X_{i_1}^2 \right) -
2 \rho_{i_1 i_2} X_{i_1} X_{i_2} - 1}{\sqrt{2}\left( 1 - \rho_{i_1 i_2}^2 \right)},\\
\psi_{\{i_1,i_2\}11}(X_{i_1},X_{i_2})=\dfrac{\sqrt{1+\rho_{i_{1}i_{2}}^2}}{\rho_{i_{1}i_{2}}^2-1}
\!\left[
\dfrac{\rho_{i_{1}i_{2}}\left(X_{i_1}^2+X_{i_2}^2\right)}{1+\rho_{i_{1}i_{2}}^2}\!-\!
X_{i_1}X_{i_2}\!+\!
\dfrac{\rho_{i_{1}i_{2}}\left(\rho_{i_{1}i_{2}}^2-1\right)}{1+\rho_{i_{1}i_{2}}^2}
\right],
\end{array}
\]
a sequence of orthonormal polynomials for quadratic approximation of any square-integrable bivariate function of $X_{i_1}$ and $X_{i_2}$.  There are multiple ways to choose a set of basis functions for $y_{\{i_1,i_2\},G}$.  The author proposes selecting a nested basis set $\{\psi_{\{i_1\}1}, \psi_{\{i_1\}2}, \psi_{\{i_2\}1}, \psi_{\{i_2\}2},\psi_{\{i_1,i_2\}11}\}$, which subsumes the basis functions for $y_{\{i_1\},G}$ and $y_{\{i_2\},G}$.  It is elementary to show that the members of such a nested basis set have \emph{zero} means and are hierarchically orthogonal, that is, $\mathbb{E}[\psi_{\{i_1,i_2\}j_1 j_2}(X_{i_1},X_{i_2})\\\psi_{\{i_1\}j_1}(X_{i_1})]=0$ and $\mathbb{E}[\psi_{\{i_1,i_2\}j_1 j_2}(X_{i_1},X_{i_2})\psi_{\{i_2\}j_2}(X_{i_2})]=0$ for $|\mathbf{j}_{2}|\le 2$, $j_1, j_2 \ne 0$.  Again, since $\mathbb{E}[y_{\{i_1,i_2\},G}(X_{i_1},X_{i_2})]=0$ following Proposition \ref{prop:3}, a constant multiplier of $\psi_{\{i_1,i_2\}11}$, excluding $\psi_{\{i_1\}1}$, $\psi_{\{i_1\}2}$, $\psi_{\{i_2\}1}$, $\psi_{\{i_2\}2}$, that is, the basis subset $\{\psi_{\{i_1,i_2\}11}\}$ is
adequate to approximate $y_{\{i_1,i_2\},G}$.  In both instances, the power $j_{k}$ for each $X_{i_k}$, $k\in u$, of the monomial $\mathbf{X}_{u}^{\mathbf{j}_{|u|}}:=X_{i_{1}}^{j_{1}} \times \cdots \times X_{i_{|u|}}^{j_{|u|}}$ in the basis subset, whether $|u|=1$ or $|u|=2$, is not \emph{zero}.  Therefore, the condition $j_{1},\cdots,j_{|u|}\neq0$ is required for selecting the basis functions of a general $|u|$-variate component function in (\ref{16}) and (\ref{17}).

The selection of nested basis functions, as explained in the preceding paragraph for $u=\{i_1,i_2\}$ and $v=\{i_1\}$ or $\{i_2\}$ and Gaussian probability measure, easily extends to a general $|u|$-variate function and general probability measure $f_{u}(\mathbf{x}_{u})d\mathbf{x}_{u}$.  Given $\emptyset\neq u\subseteq\{1,\cdots,N\}$, let $\{\psi_{v\mathbf{k}_{|v|}}(\mathbf{X}_{v}), \emptyset\ne v \subseteq u, \mathbf{k}_{|v|}\in\mathbb{N}_{0}^{|v|}, k_{1},\cdots,k_{|v|} \neq 0  \}$
be a nested set of measure-consistent orthonormal polynomial basis functions for $y_{u,G}$, which comprises as a subset measure-consistent orthonormal polynomial basis functions for
$y_{v,G}$, $\emptyset\ne v \subset u$.  Then, from fundamental properties of multivariate orthogonal polynomials, (1) $\psi_{u\mathbf{j}_{|u|}}(\mathbf{X}_{u})$ has a \emph{zero} mean for any $\emptyset\ne u\subseteq\{1,\cdots,N\}$ and $j_{1},\cdots,j_{|u|} \neq 0$; and (2) $\psi_{u\mathbf{j}_{|u|}}(\mathbf{X}_{u})$ is orthogonal to $\psi_{v\mathbf{k}_{|v|}}(\mathbf{X}_{v})$ for any $\emptyset\ne v \subset u$, $j_{1},\cdots,j_{|u|} \neq 0$, and $k_{1},\cdots,k_{|v|} \neq 0$.  Therefore, the \emph{zero}-mean property of $y_{u,G}$ and hierarchical orthogonality between $y_{u,G}$ and $y_{v,G}$, $\emptyset\ne v \subset u$, as required by Propositions \ref{prop:3} and \ref{prop:4}, are naturally satisfied.

The constant function $y_{\emptyset,G}$ defined in (\ref{11a}) is an $N$-dimensional integral, which must be calculated or estimated by some means. The evaluation of non-constant component
functions $y_{u,G}(\mathbf{x}_{u})$, $\emptyset\ne u\subseteq\{1,\cdots,N\}$,
requires calculation of the expansion coefficients defined in (\ref{17}),
which are similar integrals on at most $\mathbb{R}^{N}$. But, since
$y_{u,G}$ is unknown, the coefficients cannot be determined from
their definitions alone. Two new results, Theorem \ref{thm:8} and
Corollary \ref{corr:9}, describe how these coefficients can be calculated
from the solution of a linear system of algebraic equations.

\begin{theorem}
\label{thm:8}
Let $y$ be a square-integrable function of $\mathbf{X}$, admitting a generalized ADD, where $\mathbf{X}=(X_{1},\cdots,X_{N})$ is an $\mathbb{R}^{N}$-valued
dependent random vector with an arbitrary non-product type joint probability
density function $f_{\mathbf{X}}:\mathbb{R}^{N}\to\mathbb{R}_{0}^{+}$ and a
marginal probability density function $f_{u}$ of $\mathbf{X}_{u}$.
Given $\emptyset\neq u\subseteq\{1,\cdots,N\}$, let
$\{\psi_{v\mathbf{k}_{|v|}}(\mathbf{X}_{v}), \emptyset\neq v \subseteq u, \mathbf{k}_{|v|}\in\mathbb{N}_{0}^{|v|}, k_{1},\cdots,k_{|v|} \neq 0  \}$
be a nested set of measure-consistent orthonormal polynomial basis functions such that $\mathbb{E}[\psi_{u\mathbf{j}_{|u|}}(\mathbf{X}_{u})]=0$ and $\mathbb{E}[\psi_{u\mathbf{j}_{|u|}}(\mathbf{X}_{u})\psi_{v\mathbf{k}_{|v|}}(\mathbf{X}_{v})]=0$ for $\emptyset\ne v \subset u$, $j_{1},\cdots,j_{|u|} \neq 0$ and $k_{1},\cdots,k_{|v|} \neq 0$. Then the expansion coefficients of the polynomial representation of non-constant component functions of $y$ in (\ref{16}) and (\ref{17}) satisfy

\begin{equation}
C_{u\mathbf{j}_{|u|}} +
\sum_{{\textstyle {\emptyset\ne v\subseteq\{1,\cdots,N\}\atop v\cap u\ne\emptyset,v\nsubseteq u}}}{\displaystyle \sum_{{\textstyle {\mathbf{k}_{|v|}\in\mathbb{N}_{0}^{|v|}\atop k_{1},\cdots,k_{|v|}\neq0}}}}
C_{v\mathbf{k}_{|v|}}J_{u\mathbf{j}_{|u|},v\mathbf{k}_{|v|}}=I_{u\mathbf{j}_{|u|}},\label{20}
\end{equation}
where the integrals
\begin{subequations}
\begin{align}
I_{u\mathbf{j}_{|u|}} & :=\int_{\mathbb{R}^{N}}y(\mathbf{x})\psi_{u\mathbf{j}_{|u|}}(\mathbf{x}_{u})f_{u}(\mathbf{x}_{u})f_{-u}(\mathbf{x}_{-u})d\mathbf{x},\label{21a}\\
J_{u\mathbf{j}_{|u|},v\mathbf{k}_{|v|}} &
:={\displaystyle \int_{\mathbb{R}^{|v\cup u|}}\psi_{u\mathbf{j}_{|u|}}(\mathbf{x}_{u})\psi_{v\mathbf{k}_{|v|}}(\mathbf{x}_{v})f_{u}(\mathbf{x}_{u})f_{v\cap-u}(\mathbf{x}_{v\cap-u})d\mathbf{x}_{v\cup u}}.\label{21b}
\end{align}
\label{21} \end{subequations}
\end{theorem}

\begin{proof}
Replace $y_{u,G}$ in (\ref{17}) with the right side of (\ref{11b})
to write
\begin{equation}
\begin{array}{rcl}
C_{u\mathbf{j}_{|u|}} & = & \int_{\mathbb{R}^{N}}y(\mathbf{x})\psi_{u\mathbf{j}_{|u|}}(\mathbf{x}_{u})f_{u}(\mathbf{x}_{u})f_{-u}(\mathbf{x}_{-u})d\mathbf{x}-\\
 &  & {\displaystyle \sum_{\emptyset\ne v\subset u}}\int_{\mathbb{R}^{|u|}}y_{v,G}(\mathbf{x}_{v})\psi_{u\mathbf{j}_{|u|}}(\mathbf{x}_{u})f_{u}(\mathbf{x}_{u})d\mathbf{x}_{u}-\\
 &  & {\displaystyle \sum_{{\textstyle {\emptyset\ne v\subseteq\{1,\cdots,N\}\atop v\cap u\ne\emptyset,v\nsubseteq u}}}{\displaystyle \int_{\mathbb{R}^{|v\cup u|}}y_{v,G}(\mathbf{X}_{v})\psi_{u\mathbf{j}_{|u|}}(\mathbf{x}_{u})f_{u}(\mathbf{x}_{u})f_{v\cap-u}(\mathbf{x}_{v\cap-u})d\mathbf{x}_{v\cup u}}},
\end{array}\label{22}
\end{equation}
where in the second line the integral associated with $v=\emptyset$,
that is,
\[
\int_{\mathbb{R}^{|u|}}y_{\emptyset,G}\psi_{u\mathbf{j}_{|u|}}(\mathbf{x}_{u})f_{u}(\mathbf{x}_{u})d\mathbf{x}_{u}
=y_{\emptyset,G}\mathbb{E}\left[\psi_{u\mathbf{j}_{|u|}}(\mathbf{X}_{u})\right]
\]
drops out for all $\emptyset\ne u\subseteq\{1,\cdots,N\}$, consistent
with the definition of orthonormal polynomials. Now substitute all
component functions of $y$ involved in (\ref{22}) with their Fourier-polynomial
expansions, as described by (\ref{16}), which results in
\begin{equation}
\begin{array}{rcl}
C_{u\mathbf{j}_{|u|}} & = & \int_{\mathbb{R}^{N}}y(\mathbf{x})\psi_{u\mathbf{j}_{|u|}}(\mathbf{x}_{u})f_{u}(\mathbf{x}_{u})f_{-u}(\mathbf{x}_{-u})d\mathbf{x}-\\
 &  & {\displaystyle \sum_{\emptyset\ne v\subset u}}{\displaystyle \sum_{{\textstyle {\mathbf{k}_{|v|}\in\mathbb{N}_{0}^{|v|}\atop k_{1},\cdots,k_{|v|}\neq0}}}}C_{v\mathbf{k}_{|v|}}\int_{\mathbb{R}^{|u|}}\psi_{u\mathbf{j}_{|u|}}(\mathbf{x}_{u})\psi_{v\mathbf{k}_{|v|}}(\mathbf{x}_{v})f_{u}(\mathbf{x}_{u})d\mathbf{x}_{u}-\\
 &  & {\displaystyle \sum_{{\textstyle {\emptyset\ne v\subseteq\{1,\cdots,N\}\atop v\cap u\ne\emptyset,v\nsubseteq u}}}{\displaystyle \sum_{{\textstyle {\mathbf{k}_{|v|}\in\mathbb{N}_{0}^{|v|}\atop k_{1},\cdots,k_{|v|}\neq0}}}}C_{v\mathbf{k}_{|v|}}}\times\\
 &  & {\displaystyle \int_{\mathbb{R}^{|v\cup u|}}\psi_{u\mathbf{j}_{|u|}}(\mathbf{x}_{u})\psi_{v\mathbf{k}_{|v|}}(\mathbf{x}_{v})f_{u}(\mathbf{x}_{u})f_{v\cap-u}(\mathbf{x}_{v\cap-u})d\mathbf{x}_{v\cup u}}.
\end{array}
\label{r6}
\end{equation}
From fundamental properties of orthogonal polynomials and nested construction of basis functions, $\psi_{u\mathbf{j}_{|u|}}(\mathbf{X}_{u})$ is orthogonal to $\psi_{v\mathbf{k}_{|v|}}(\mathbf{X}_{v})$, that is, the expectation or the integral
\begin{equation*}
\mathbb{E}\left[\psi_{u\mathbf{j}_{|u|}}(\mathbf{X}_{u})\psi_{v\mathbf{k}_{|v|}}(\mathbf{X}_{v})\right]=
\int_{\mathbb{R}^{|u|}}\psi_{u\mathbf{j}_{|u|}}(\mathbf{x}_{u})\psi_{v\mathbf{k}_{|v|}}(\mathbf{x}_{v})f_{u}(\mathbf{x}_{u})d\mathbf{x}_{u}=0
\end{equation*}
for any $\emptyset\ne v \subset u$, $j_{1},\cdots,j_{|u|} \neq 0$, and $k_{1},\cdots,k_{|v|} \neq 0$.
Therefore, (\ref{r6}) reduces to
\begin{equation}
\begin{array}{rcl}
C_{u\mathbf{j}_{|u|}} & = & \int_{\mathbb{R}^{N}}y(\mathbf{x})\psi_{u\mathbf{j}_{|u|}}(\mathbf{x}_{u})f_{u}(\mathbf{x}_{u})f_{-u}(\mathbf{x}_{-u})d\mathbf{x}-\\
 &  & {\displaystyle \sum_{{\textstyle {\emptyset\ne v\subseteq\{1,\cdots,N\}\atop v\cap u\ne\emptyset,v\nsubseteq u}}}{\displaystyle \sum_{{\textstyle {\mathbf{k}_{|v|}\in\mathbb{N}_{0}^{|v|}\atop k_{1},\cdots,k_{|v|}\neq0}}}}C_{v\mathbf{k}_{|v|}}}\times\\
 &  & {\displaystyle \int_{\mathbb{R}^{|v\cup u|}}\psi_{u\mathbf{j}_{|u|}}(\mathbf{x}_{u})\psi_{v\mathbf{k}_{|v|}}(\mathbf{x}_{v})f_{u}(\mathbf{x}_{u})f_{v\cap-u}(\mathbf{x}_{v\cap-u})d\mathbf{x}_{v\cup u}}.
\end{array}\label{23b}
\end{equation}
Defining integrals $I_{u\mathbf{j}_{|u|}}$ and $J_{u\mathbf{j}_{|u|},v\mathbf{k}_{|v|}}$,
as in (\ref{21a}) and (\ref{21b}), (\ref{23b}) simplifies to (\ref{20}),
proving the theorem.\end{proof}

\subsection{Finite-Dimensional Approximation}

Equation (\ref{20}) describes an infinite-dimensional system involving an infinite number of coefficients. Therefore, a finite-dimensional approximation, leading to approximate expansion coefficients and a truncated generalized ADD, must be used in practice.  Corollary \ref{corr:9} provides such a solution.

\begin{corollary}
\label{corr:9}When truncated at $|u|=S$ and $|\mathbf{j}_{|u|}|:=j_{1}+\cdots+j_{|u|}= m$, where $1\le S < N$ and $S\le m < \infty$, the approximate expansion coefficients $\tilde{C}_{u\mathbf{j}_{|u|}}$ satisfy

\begin{equation}
\begin{array}{c}
\tilde{C}_{u\mathbf{j}_{|u|}}+\!\!{\displaystyle
\!\!\sum_{{\textstyle {\emptyset\ne v\subseteq\{1,\cdots,N\}\atop v\cap u\ne\emptyset,v\nsubseteq u}}}}{\displaystyle
\sum_{k=1}^{m}}{\displaystyle
\!\!\sum_{{\textstyle {|\mathbf{k}_{|v|}|=k\atop k_{1},\cdots,k_{|v|}\neq0}}}}\!\!\!\!\!\tilde{C}_{v\mathbf{k}_{|v|}}J_{u\mathbf{j}_{|u|},v\mathbf{k}_{|v|}}=I_{u\mathbf{j}_{|u|}},
~1 \le |u| \le S,~1 \le |\mathbf{j}_{|u|}| \le m,
\label{24}
\end{array}
\end{equation}
and
\begin{equation}
\tilde{C}_{u\mathbf{j}_{|u|}}=0,~1 \le |u| \le S, ~m+1 \le |\mathbf{j}_{|u|}| < \infty;
~S+1 \le |u| \le N, ~1 \le |\mathbf{j}_{|u|}| < \infty.
\label{24b}
\end{equation}

\end{corollary}

Expressed compactly, the system of equations in (\ref{24}) for non-trivial solutions of $\tilde{C}_{u\mathbf{j}_{|u|}}$
forms an $L_{S,m}\times L_{S,m}$ matrix equation: $\mathbf{A}\mathbf{z}=\mathbf{b}$,
where $\mathbf{A}\in\mathbb{R}^{L_{S,m}\times L_{S,m}}$ contains integrals $J_{u\mathbf{j}_{|u|},v\mathbf{k}_{|v|}}$,
$\mathbf{b}\in\mathbb{R}^{L_{S,m}}$ comprises integrals $I_{u\mathbf{j}_{|u|}}$,
and $\mathbf{z}\in\mathbb{R}^{L_{S,m}}$ is the solution vector of the approximate expansion
coefficients. The size of the matrix equation is
\[
L_{S,m}=\sum_{k=1}^{N}\binom{N}{k}\binom{m}{k},
\]
where $\binom{m}{k}=0$ if $k>m$, when the measure-consistent orthonormal polynomials are constructed by satisfying the condition  $j_{1},\cdots,j_{|u|}\neq0$, as explained previously.  The matrix form of (\ref{24}) is easy to implement and solve, and is scalable to higher dimensions in a straightforward way.  It is elementary to show that $\tilde{C}_{u\mathbf{j}_{|u|}} \to C_{u\mathbf{j}_{|u|}}$ as $S \to N$ and $m \to\infty$.

The truncations introduced in Corollary \ref{corr:9} engender an $S$-variate, $m$th-order generalized ADD approximation
\begin{equation}
\tilde{y}_{S,m}(\mathbf{X})=y_{\emptyset,G}+
{\displaystyle\sum_{{\textstyle {\emptyset\ne u\subseteq\{1,\cdots,N\}\atop 1 \le |u| \le S}}}}
{\displaystyle \sum_{k=1}^{m}}{\displaystyle \sum_{{\textstyle {|\mathbf{j}_{|u|}|=k\atop j_{1},\cdots,j_{|u|}\neq0}}}}
\tilde{C}_{u\mathbf{j}_{|u|}}\psi_{u\mathbf{j}_{|u|}}
\label{r7}
\end{equation}
of $y(\mathbf{X})$ in (\ref{7}), which is grounded on a fundamental conjecture known to be true in many real-world applications: given a high-dimensional function $y$, its $|u|$-variate component functions decay rapidly with respect to $|u|$, leading to accurate lower-variate approximations of $y$.  For instance, by selecting $S=1$ or $2$, the functions $\tilde{y}_{1,m}$ and $\tilde{y}_{2,m}$, respectively, provide univariate and bivariate approximations of $y$.  The higher the value of $S$ and/or $m$, the higher the accuracy, but also the concomitant computational effort.  When $S\to N$ and $m\to \infty$, $\tilde{y}_{S,m}$ converges to $y$ in the mean-square sense, generating a hierarchical and convergent sequence of approximations.

The computational effort in determining the expansion coefficients $y_{\emptyset,G}$ and $\tilde{C}_{u\mathbf{j}_{|u|}}$ is rooted in efficient and accurate calculations of various $N$-dimensional integrals, including $I_{u\mathbf{j}_{|u|}}$, $u\subseteq\{1,\cdots,N\}$, $1 \le |u| \le S$, $|\mathbf{j}_{|u|}|\le m$.  For large $N$, a full numerical integration employing an $N$-dimensional tensor product of a univariate quadrature rule is computationally prohibitive.  Instead, a dimension-reduction integration scheme, developed by Xu and Rahman \cite{xu04}, can be applied to estimate the coefficients efficiently.  The scheme entails approximating a high-dimensional integral of interest by a finite sum of lower-dimensional integrations.  The computational complexity is $S$th-order polynomial $-$ for instance, linear or quadratic when $S=1$ or 2 $-$ with respect to the number of variables or integration points, alleviating the curse of dimensionality to an extent determined by $S$.  See the work of Xu and Rahman \cite{xu04} for further details.

If the function $y$ is a sum of at most $S$-variate, $m$th-order polynomials, then (\ref{24}) yields the exact solution of the expansion coefficients $C_{u\mathbf{j}_{|u|}}$ and (\ref{r7}) exactly reproduces $y$, provided that the integrals $I_{u\mathbf{j}_{|u|}}$ and $J_{u\mathbf{j}_{|u|},v\mathbf{k}_{|v|}}$ are calculated exactly. Numerical results corroborating theoretical findings will be presented in Section 7.

\section{Second-Moment Analysis}

Once the component functions of the generalized ADD have been determined,
subsequent evaluations of their second-moment characteristics, including
global sensitivity analysis, are conducted as follows.

\subsection{Mean and Variance}

Applying the expectation operator on (\ref{7}) and noting Proposition
\ref{prop:3}, the mean
\begin{equation}
\mu:=\mathbb{E}[y(\mathbf{X})]=y_{\emptyset,G}\label{26}
\end{equation}
of $y(\mathbf{X})$ matches the constant component function of the
generalized ADD. This is similar to (\ref{5}), the result from the
classical ADD, although the respective constants involved are not
the same. Applying the expectation operator again, this time on $(y(\mathbf{X})-\mu)^{2}$,
and recognizing Proposition \ref{prop:4} results in the variance
\begin{equation}
\sigma^{2}:=\mathbb{E}\left[\left(y(\mathbf{X})-\mu\right)^{2}\right]={\displaystyle \sum_{\emptyset\neq u\subseteq\{1,\cdots,N\}}\!\!\!\!\!\!\!\!\!\!\mathbb{E}\left[y_{u,G}^{2}(\mathbf{X}_{u})\right]}+\!\sum_{{\textstyle {\emptyset\ne u,v\subseteq\{1,\cdots,N\}\atop u \nsubseteq v\nsubseteq u}}}\!\!\!\!\!\!\!\!\!\!\mathbb{E}\left[y_{u,G}(\mathbf{X}_{u})y_{v,G}(\mathbf{X}_{v})\right]\label{27}
\end{equation}
of $y(\mathbf{X})$, where the first sum represents variance contributions
from all non-constant component functions. In contrast, the second
sum in (\ref{27}) typifies covariance contributions from two distinct
non-constant component functions that are not orthogonal $-$ a ramification
of imposing the weak annihilating conditions appropriate for the generalized
ADD. The latter sum disappears altogether in the classical ADD because
of the strong annihilating conditions that are possible to enforce
for independent probability measures. Nonetheless, (\ref{26}) and
(\ref{27}) furnish new generalized formulae for the second-moment
statistics of $y(\mathbf{X})$ in terms of the moments of relevant
component functions.

\subsection{Global Sensitivity Indices}

Mathematical modeling of complex systems often requires sensitivity
analysis to determine how an output variable of interest is influenced
by individual or subsets of input variables. A global sensitivity
analysis constitutes the study of how the output variance from a mathematical
model is divvied up, qualitatively or quantitatively, to distinct
sources of input variation in the model \cite{sobol01}. There exist a multitudes of methods or techniques for calculating the global sensitivity indices of a function of independent variables: the random balance design method \cite{tarantola06}, state-dependent parameter metamodel \cite{ratto07}, Sobol's method \cite{sobol93}, polynomial chaos expansion \cite{sudret08}, polynomial dimensional decomposition \cite{rahman11}, non-parametric regression procedures \cite{storlie09}, and many others. In contrast, only a few methods, such as those presented by Li \emph{et al}. \cite{li12}, Kucherenko \emph{et al}. \cite{kucherenko12}, and Chastaing \emph{et al} \cite{chastaing12}, are available for models with dependent or correlated input.  In this section,
a triplet of global sensitivity indices is defined for problems involving
dependent probability distributions of input variables.

Three $|u|$-variate global sensitivity indices of a stochastic response
function $y(\mathbf{X})$ for a subset $\mathbf{X}_{u}$ of input
variables $\mathbf{X}$, denoted by $S_{u,v}$, $S_{u,c}$, and $S_{u}$,
are defined as the ratios
\begin{equation}
S_{u,v}:=\frac{\mathbb{E}\left[y_{u,G}^{2}(\mathbf{X}_{u})\right]}{\sigma^{2}},\label{28}
\end{equation}
\begin{equation}
S_{u,c}:=\frac{{\displaystyle \sum_{{\textstyle {\emptyset\ne v\subseteq\{1,\cdots,N\}\atop u \nsubseteq v\nsubseteq u}}}}\mathbb{E}\left[y_{u,G}(\mathbf{X}_{u})y_{v,G}(\mathbf{X}_{v})\right]}{\sigma^{2}},\label{29}
\end{equation}
\begin{equation}
S_{u}:=S_{u,v}+S_{u,c},\label{30}
\end{equation}
provided that the variance $0<\sigma^{2}<\infty$. The first two indices
$S_{u,v}$ and $S_{u,c}$ represent the normalized versions of the
variance contribution from $y_{u,G}$ to $\sigma^{2}$ and of the
covariance contributions from $y_{u,G}$ and all $y_{v,G}$, such
that $u \nsubseteq v\nsubseteq u$, to $\sigma^{2}$. They will be named the
variance-driven global sensitivity index and the covariance-driven
global sensitivity index, respectively, of $y(\mathbf{X})$ for $\mathbf{X}_{u}$.
The third index $S_{u}$, referred to as the total global sensitivity
index of $y(\mathbf{X})$ for $\mathbf{X}_{u}$, is the sum of variance
and covariance contributions from or associated with $y_{u,G}$ to
$\sigma^{2}$. Since $\emptyset\ne u\subseteq\{1,\cdots,N\}$, there
exist $2^{N}-1$ such triplets of indices, adding up to
\begin{equation}
{\displaystyle \sum_{\emptyset\neq u\subseteq\{1,\cdots,N\}}S_{u}}={\displaystyle \sum_{\emptyset\neq u\subseteq\{1,\cdots,N\}}S_{u,v}}+{\displaystyle \sum_{\emptyset\neq u\subseteq\{1,\cdots,N\}}S_{u,c}}=1.\label{31}
\end{equation}

From the definitions, the variance-driven sensitivity index $S_{u,v}$ is a non-negative, real-valued number.  It reflects the contribution of $\mathbf{X}_{u}$ through $y_u(\mathbf{X}_{u})$ in the system structure of $y(\mathbf{X})$.  In contrast, the covariance-driven sensitivity index $S_{u,v}$  can be negative, positive, or zero, depending on the correlation between $\mathbf{X}_{u}$ and $\mathbf{X}_{v}$.  It represents the contribution of $\mathbf{X}_{u}$ by the interaction of $y_u(\mathbf{X}_{u})$  and $y_v(\mathbf{X}_{v})$, when $u \nsubseteq v\nsubseteq u$, due to dependent probability distribution. Depending on whether $S_{u,c}$ is positive or negative, $S_{u,c}$ strengthens or weakens $S_{u}$, provided that $S_{u}>0$.  The individual sums of these two indices in (\ref{31}) over all $\emptyset\neq u\subseteq\{1,\cdots,N\}$ may exceed unity or be negative, but the sum of these two individual
sums is always equal to \emph{one}. When the random variables are
independent, the covariance-driven contribution to the total sensitivity
index vanishes, leaving behind only one sensitivity index for the
classical ADD. The global sensitivity indices, whether derived from
the generalized or classical ADD, can be used to rank variables, fix
unessential variables, and reduce dimensions of large-scale problems. They also facilitate a means to define effective dimensions of the function $y$, as follows.

Li \emph{et al}. \cite{li10} have presented similar definitions of the three sensitivity indices under the names, structural, correlative, and total sensitivity indices.  However, their correlative sensitivity index, defined by
\begin{equation}
S_{u,c}':=\frac{{\displaystyle \sum_{{\textstyle {\emptyset\ne v\subseteq\{1,\cdots,N\}\atop v\neq u}}}}\mathbb{E}\left[y_{u,G}(\mathbf{X}_{u})y_{v,G}(\mathbf{X}_{v})\right]}{\sigma^{2}},
\label{r8}
\end{equation}
represents an unreduced version of the covariance-driven sensitivity index $S_{u,c}$ defined in (\ref{29}). The difference between these two definitions stems from not recognizing the hierarchical orthogonality condition.  Indeed, using Proposition \ref{prop:4}, the condition $u\neq v$ reduces to $u \nsubseteq v\nsubseteq u$, resulting in $S_{u,c}'=S_{u,c}$.

\subsection{Effective Dimensions}

For many practical applications, the multivariate function $y$ of
$N$ variables, fortunately, can be effectively approximated by a
sum of at most $S$-variate component functions $y_{u,G}$, $1\le|u|\le S\le N$,
of the generalized ADD in (\ref{7}). The truncation can be achieved
by the notion of effective dimension, introduced by Caflisch \emph{et
al}. \cite{caflisch97}, who exploited the classical ADD-based low effective dimension
to explain why the quasi Monte Carlo method outperforms the crude
Monte Carlo algorithm for evaluating a certain class of high-dimensional
integrals. In this section, extended definitions of two generalized
effective dimensions, stemming from the generalized ADD and global
sensitivity indices, are presented.

\begin{definition} \label{def:11}A square-integrable multivariate
function $y$ of $\mathbf{X}\in\mathbb{R}^{N}$ with finite variance
$0<\sigma^{2}<\infty$ has a generalized effective dimension $1\le S_{s}\le N$
in the superposition sense, henceforth denoted as the superposition
dimension, if
\[
S_{s}:=\min\left\{ S:1\le S\le N\;\text{such\;\ that\;} \left|1-\sum_{{\textstyle {\emptyset\ne u\subseteq\{1,\cdots,N\}\atop 1\le|u|\le S}}}S_{u}\right| \le 1-p  \right\}
\]
and a generalized effective dimension $1\le S_{t}\le N$ in the truncation
sense, henceforth denoted as the truncation dimension, if
\[
S_{t}:=\min\left\{ S:1\le S\le N\;\text{such\;\ that\;} \left|1-\sum_{u\subseteq\{1,\cdots,S\}}S_{u}\right| \le 1-p \right\} ,
\]
where $S_{u}$ is the total global sensitivity index of $y(\mathbf{X})$
for $\mathbf{X}_{u}$, $\emptyset\ne u\subseteq\{1,\cdots,N\}$, and
$0\le p\le1$ is a percentile threshold.
\end{definition}

Caflisch \emph{et al}. used the 99th percentile for $p$, but it can
be treated as a threshold parameter linked to the desired accuracy
of a stochastic solution. Both definitions capture the notion in which
the function $y$ is almost $S_{s}$- or $S_{t}$-dimensional. The
relevance of the truncation or superposition dimension depends on
the nature of the function. The former signifies the number of important
random variables and is appropriate when some variables are more important
than others in an ordered set. In contrast, the latter determines
whether the low-variate component functions of dimensional decomposition
dominate the function and is appropriate when all variables are
equally important. For truly high-dimensional problems, all variables
contribute to a function value; therefore, the superposition dimension
is more useful than the truncation dimension.

According to the definitions, evaluations of the generalized effective dimensions require calculating the variance $\sigma^2$ exactly, which is infeasible, if not impossible, for a general function $y$ of an arbitrary number of variables.  However, the $S$-variate, $m$th-order approximation $\tilde{y}_{S,m}(\mathbf{X})$, discussed in conjunction with Corollary \ref{corr:9}, can be used to estimate the variance of $y(\mathbf{X})$, furnishing a practical means to calculate the effective dimensions. In which case, a convergence analysis with respect to $S$ and $m$, or an adaptive version of (\ref{r7}), briefly described in Section 6.4, will be required.

\subsection{Adaptive-Sparse Approximation}
The global sensitivity indices can be exploited to create an adaptive-sparse ADD approximation of a high-dimensional function. Let $\epsilon_{1}\ge0$
and $\epsilon_{2}\ge0$ denote two non-negative error tolerances that specify the minimum values of $\tilde{S}_{u,m_{u}}$ , which is the $m_{u}$th-order approximation of $S_{u}$, and $\Delta\tilde{S}_{u,m_{u}}:=(\tilde{S}_{u,m_{u}}-\tilde{S}_{u,m_{u}-1})/\tilde{S}_{u,m_{u}-1}$, provided that $\tilde{S}_{u,m_{u}-1}\ne 0$ . Then an adaptive-sparse ADD approximation
\begin{equation}
\bar{y}(\mathbf{X}):=y_{\emptyset,G}+{\displaystyle \sum_{\emptyset\ne u\subseteq\{1,\cdots,N\}}}~{\displaystyle \sum_{m_{u}=1}^{\infty}}~\sum_{{\textstyle {\left| \mathbf{j}_{|u|}\right|=m_{u},\, j_{1},\cdots,j_{|u|}\neq0\atop \tilde{S}_{u,m_{u}}>\epsilon_{1},\Delta\tilde{S}_{u,m_{u}}>\epsilon_{2}}}}C_{u\mathbf{j}_{|u|}}\psi_{u\mathbf{j}_{|u|}}(\mathbf{X}_{u})
\nonumber
\end{equation}
of $y(\mathbf{X})$ is formed by the subset of ANOVA component functions, satisfying two inclusion criteria: (1) $\tilde{S}_{u,m_{u}}>\epsilon_{1}$, and (2) $\Delta\tilde{S}_{u,m_{u}}>\epsilon_{2}$
for all $1\le|u|\le N$ and $1\le m_{u}<\infty$. The first criterion requires the contribution of an $m_{u}$-th order polynomial approximation of $y_{u,G}(\mathbf{X}_{u})$ towards the variance of $y(\mathbf{X})$ to exceed $\epsilon_{1}$ in order to be accommodated within the resultant truncation. The second criterion identifies the augmentation in the variance contribution from $y_{u,G}(\mathbf{X}_{u})$ evoked by a single increment in the polynomial order $m_{u}$ and determines if it surpasses $\epsilon_{2}$.  In other words, these two criteria ascertain which interactive effects between two or more input random variables are retained and dictate the largest order of polynomials in a component function, formulating a fully adaptive-sparse ANOVA approximation. No truncation parameters, that is, $S,m$ of the truncated ADD need to be selected \emph{apriori} or arbitrarily. Although successfully developed for independent variables \cite{yadav14}, additional efforts are needed for numerical implementation of the adaptive-sparse approximation for dependent variables.

\section{Examples}
Two illustrative examples, the one entailing an explicit mathematical function and the other involving an implicit function derived from finite-element analysis, are presented.

\subsection{A Mathematical Function}
Consider a quadratic polynomial function
\begin{equation}
y=(a_{0}+a_{1}X_{1})(b_{0}+b_{1}X_{2})+(a_{0}+a_{1}X_{1})(c_{0}+c_{1}X_{3})+(b_{0}+b_{1}X_{2})(c_{0}+c_{1}X_{3})
\nonumber
\end{equation}
of a trivariate Gaussian random vector $\mathbf{X}=(X_{1},X_{2},X_{3})^{T}\in\mathbb{R}^{3}$,
which has mean $\mathbb{E}[\mathbf{X}]=\mathbf{0}\in\mathbb{R}^{3}$,
positive-definite covariance matrix
\[
\mathbf{\Sigma}_{\mathbf{X}}=\mathbb{E}\left[\mathbf{X}\mathbf{X}{}^{T}\right]=\left[\begin{array}{ccc}
\sigma_{1}^{2} & \rho_{12}\sigma_{1}\sigma_{2} & \rho_{13}\sigma_{1}\sigma_{3}\\
 & \sigma_{2}^{2} & \rho_{23}\sigma_{2}\sigma_{3}\\
(\mathrm{sym}.) &  & \sigma_{3}^{2}
\end{array}\right]\in\mathbb{R}^{3\times3},
\]
comprising variances $\sigma_{i}^{2}=1$ of $X_{i}$ for $i=1,2,3$
and correlation coefficients $\rho_{ij}$ between $X_{i}$ and $X_{j}$,
$i,j=1,2,3$, $i\ne j$, and joint probability density function described by (\ref{r9}) for $N=3$.
Four sets of correlation coefficients with varied strengths and types
of statistical dependence among random variables were examined: (1)
$\rho_{12}=\rho_{13}=\rho_{23}=0$ (uncorrelated); (2) $\rho_{12}=\rho_{13}=\rho_{23}=1/5$
(equally correlated); (3) $\rho_{12}=1/5$, $\rho_{13}=2/5$, $\rho_{23}=4/5$
(positively correlated); and (4) $\rho_{12}=-1/5$, $\rho_{13}=2/5$,
$\rho_{23}=-4/5$ (mixedly correlated). The deterministic parameters
are $a_{0}=b_{0}=c_{0}=2$, $a_{1}=b_{1}=c_{1}=1$, rendering $y$
a symmetric function. The objective of this simple yet insightful
example is to explain how the proposed methods can be applied to determine
the component functions of and global sensitivity indices from the
generalized ADD.

Given the Gaussian probability density function of $\mathbf{X}$,
the marginal probability densities of $\mathbf{X}_{u}$, $\emptyset\ne u\subseteq\{1,2,3\}$,
are also Gaussian, and is described by (\ref{r10}). The probability density
$\phi_{u}(\mathbf{x}_{u};\mathbf{\Sigma}_{u})$ induces multivariate Hermite orthonormal polynomials $\{\psi_{u\mathbf{j}_{|u|}}\}$, as described by (\ref{r4}) and (\ref{r5}), that are consistent with the probability measure of $\mathbf{X}_{u}$.
From these orthonormal polynomials and the function $y$, the integrals
$I_{u\mathbf{j}_{|u|}}$ and $J_{u\mathbf{j}_{|u|},v\mathbf{k}_{|v|}}$
were exactly calculated or determined from their definitions in (\ref{21a}) and
(\ref{21b}). Using Corollary \ref{corr:9}, that is, (\ref{24}), $S=2$,
$m=2$, and these two sets of integrals, a system of linear equations
was generated and then solved to determine exactly the expansion
coefficients $\tilde{C}_{u\mathbf{j}_{|u|}}$ for $\emptyset\ne u\subseteq\{1,2,3\}$
and $|\mathbf{j}_{|u|}|\le m$. Since $y$ is a sum of at most bivariate, second-order polynomials,
the selection of $S=2$ and $m=2$ is adequate to produce $\tilde{C}_{u\mathbf{j}_{|u|}}=C_{u\mathbf{j}_{|u|}}$, thereby exactly reproducing $y$ from the generalized ADD.

Table \ref{table1} presents all eight component functions of $y$,
obtained exactly using the proposed methods in Sections 5 and 6, for
four distinct cases of correlation properties of $\mathbf{X}$. It
is elementary to verify that all component functions have zero means
(Proposition \ref{prop:1} or \ref{prop:3}) and are either fully
orthogonal (Proposition \ref{prop:2}) for case 1 or hierarchically
orthogonal (Proposition \ref{prop:4}) for cases 2 through 4. When there is no
correlation between any two random variables, that is, $\rho_{12}=\rho_{13}=\rho_{23}=0$,
the proposed method replicates exactly the component functions of
the classical ADD. Clearly, the component functions vary with the
correlation structure, but when added together they reconstruct the
function $y$ regardless of whether or not the random variables are independent. It is important to note that the univariate parts of $y$, which are strictly linear functions of $X_{i}$, are exactly reproduced
only when there is no correlation between any two random variables, that
is, when invoking the classical ADD. In contrast, the univariate component
functions derived from the generalized ADD with a non-trivial correlation
structure contain second-order terms and are generally nonlinear.
This is due to statistical dependence among random variables, inducing
higher-order univariate terms that are not present in the original
function to begin with. A similar behavior is observed when comparing
the bivariate component functions in Table \ref{table1}. The additional
higher-order terms generated by dependent probability measures vanish
when summing all component functions of a generalized ADD. The mean
and variance of $y$ for all four cases, calculated using (\ref{26})
and (\ref{27}), are also displayed in Table \ref{table1}.

\begin{table}
\caption{The generalized ADD component functions and second-moment statistics
of $y$$^{(\text{a})}$ }

\begin{centering}
\def\arraystretch{1.4}
\begin{tabular}{c|c|c}
\hline
{\footnotesize Case} & {\footnotesize Component functions}  & {\footnotesize Moments} \tabularnewline
\hline
 & {\footnotesize $y_{\emptyset,G}=y_{\emptyset,C}=12$}  & \tabularnewline
 & {\footnotesize $y_{\{1\},G}=y_{\{1\},C}=4X_{1}$}  & \tabularnewline
 & {\footnotesize $y_{\{2\},G}=y_{\{2\},C}=4X_{2}$} & \tabularnewline
{\footnotesize Case 1: Uncorrelated}  & {\footnotesize $y_{\{3\},G}=y_{\{3\},C}=4X_{3}$} & {\footnotesize $\mu={\displaystyle 12}$} \tabularnewline
{\footnotesize $(\rho_{12}=\rho_{13}=\rho_{23}=0)$}  & {\footnotesize $y_{\{1,2\},G}=y_{\{1,2\},C}=X_{1}X_{2}$} & {\footnotesize $\sigma^{2}={\displaystyle 51}$} \tabularnewline
 & {\footnotesize $y_{\{1,3\},G}=y_{\{1,3\},C}=X_{1}X_{3}$} & \tabularnewline
 & {\footnotesize $y_{\{2,3\},G}=y_{\{2,3\},C}=X_{2}X_{3}$} & \tabularnewline
 & {\footnotesize $y_{\{1,2,3\},G}=y_{\{1,2,3\},C}=0$} & \tabularnewline
\hline
 & {\footnotesize $y_{\emptyset,G}={\displaystyle \frac{63}{5}}$}  & \tabularnewline
 & {\footnotesize $y_{\{1\},G}=-{\displaystyle \frac{5}{13}}+4X_{1}+{\displaystyle \frac{5}{13}}X_{1}^{2}$}  & \tabularnewline
 & {\footnotesize $y_{\{2\},G}=-{\displaystyle \frac{5}{13}}+4X_{2}+{\displaystyle \frac{5}{13}}X_{2}^{2}$} & \tabularnewline
{\footnotesize Case 2: Equally correlated} & {\footnotesize $y_{\{3\},G}=-{\displaystyle \frac{5}{13}}+4X_{3}+{\displaystyle \frac{5}{13}}X_{3}^{2}$} & {\footnotesize $\mu={\displaystyle \frac{63}{5}}$} \tabularnewline
{\footnotesize $(\rho_{12}=\rho_{13}=\rho_{23}=1/5)$}  & {\footnotesize $y_{\{1,2\},G}={\displaystyle \frac{12}{65}}-{\displaystyle \frac{5}{26}}X_{1}^{2}+X_{1}X_{2}-{\displaystyle \frac{5}{26}}X_{2}^{2}$} & {\footnotesize $\sigma^{2}={\displaystyle \frac{1794}{25}}$} \tabularnewline
 & {\footnotesize $y_{\{1,3\},G}={\displaystyle \frac{12}{65}}-{\displaystyle \frac{5}{26}}X_{1}^{2}+X_{1}X_{3}-{\displaystyle \frac{5}{26}}X_{3}^{2}$} & \tabularnewline
 & {\footnotesize $y_{\{2,3\},G}={\displaystyle \frac{12}{65}}-{\displaystyle \frac{5}{26}}X_{2}^{2}+X_{2}X_{3}-{\displaystyle \frac{5}{26}}X_{3}^{2}$} & \tabularnewline
 & {\footnotesize $y_{\{1,2,3\},G}=0$} & \tabularnewline
\hline
 & {\footnotesize $y_{\emptyset,G}={\displaystyle \frac{67}{5}}$}  & \tabularnewline
 & {\footnotesize $y_{\{1\},G}=-{\displaystyle \frac{405}{754}}+4X_{1}+{\displaystyle \frac{405}{754}}X_{1}^{2}$}  & \tabularnewline
 & {\footnotesize $y_{\{2\},G}=-{\displaystyle \frac{725}{1066}}+4X_{2}+{\displaystyle \frac{725}{1066}}X_{2}^{2}$} & \tabularnewline
{\footnotesize Case 3: Positively correlated} & {\footnotesize $y_{\{3\},G}=-{\displaystyle \frac{990}{1189}}+4X_{3}+{\displaystyle \frac{990}{1189}}X_{3}^{2}$} & {\footnotesize $\mu={\displaystyle \frac{67}{5}}$} \tabularnewline
{\footnotesize $(\rho_{12}=1/5,\rho_{13}=2/5,\rho_{23}=4/5)$} & {\footnotesize $y_{\{1,2\},G}={\displaystyle \frac{12}{65}}-{\displaystyle \frac{5}{26}}X_{1}^{2}+X_{1}X_{2}-{\displaystyle \frac{5}{26}}X_{2}^{2}$} & {\footnotesize $\sigma^{2}={\displaystyle \frac{2514}{25}}$} \tabularnewline
 & {\footnotesize $y_{\{1,3\},G}={\displaystyle \frac{42}{145}}-{\displaystyle \frac{10}{29}}X_{1}^{2}+X_{1}X_{3}-{\displaystyle \frac{10}{29}}X_{3}^{2}$} & \tabularnewline
 & {\footnotesize $y_{\{2,3\},G}={\displaystyle \frac{36}{205}}-{\displaystyle \frac{20}{41}}X_{2}^{2}+X_{2}X_{3}-{\displaystyle \frac{20}{41}}X_{3}^{2}$} & \tabularnewline
 & {\footnotesize $y_{\{1,2,3\},G}=0$} & \tabularnewline
\hline
 & {\footnotesize $y_{\emptyset,G}={\displaystyle \frac{57}{5}}$}  & \tabularnewline
 & {\footnotesize $y_{\{1\},G}=-{\displaystyle \frac{115}{754}}+4X_{1}+{\displaystyle \frac{115}{754}}X_{1}^{2}$}  & \tabularnewline
 & {\footnotesize $y_{\{2\},G}={\displaystyle \frac{725}{1066}}+4X_{2}-{\displaystyle \frac{725}{1066}}X_{2}^{2}$} & \tabularnewline
{\footnotesize Case 4: Mixedly correlated} & {\footnotesize $y_{\{3\},G}={\displaystyle \frac{170}{1189}}+4X_{3}-{\displaystyle \frac{170}{1189}}X_{3}^{2}$} & {\footnotesize $\mu={\displaystyle \frac{57}{5}}$} \tabularnewline
{\footnotesize $(\rho_{12}=-1/5,\rho_{13}=2/5,\rho_{23}=-4/5)$} & {\footnotesize $y_{\{1,2\},G}=-{\displaystyle \frac{12}{65}}+{\displaystyle \frac{5}{26}}X_{1}^{2}+X_{1}X_{2}+{\displaystyle \frac{5}{26}}X_{2}^{2}$} & {\footnotesize $\sigma^{2}={\displaystyle \frac{774}{25}}$} \tabularnewline
 & {\footnotesize $y_{\{1,3\},G}={\displaystyle \frac{42}{145}}-{\displaystyle \frac{10}{29}}X_{1}^{2}+X_{1}X_{3}-{\displaystyle \frac{10}{29}}X_{3}^{2}$} & \tabularnewline
 & {\footnotesize $y_{\{2,3\},G}=-{\displaystyle \frac{36}{205}}+{\displaystyle \frac{20}{41}}X_{2}^{2}+X_{2}X_{3}+{\displaystyle \frac{20}{41}}X_{3}^{2}$} & \tabularnewline
 & {\footnotesize $y_{\{1,2,3\},G}=0$} & \tabularnewline
\hline
\end{tabular}
\par\end{centering}

{\footnotesize (a) $y=12+4X_{1}+4X_{2}+4X_{3}+X_{1}X_{2}+X_{1}X_{3}+X_{2}X_{3}$,
where $a_{0}=b_{0}=c_{0}=2$, $a_{1}=b_{1}=c_{1}=1$.}{\footnotesize \par}

\label{table1}
\end{table}

The component functions listed in Table \ref{table1} were employed
for calculating the variance-driven, covariance-driven, and total
global sensitivity indices defined in (\ref{28}), (\ref{29}), and
(\ref{30}). The expectations involved in (\ref{28}) and (\ref{29})
were exactly evaluated from their respective integral definitions.
Table \ref{table2} enumerates the triplets of sensitivity indices,
$S_{u,v}$, $S_{u,c}$, and $S_{u}$, of $y$ for $X_{1}$, $X_{2}$,
$X_{3}$, $(X_{1},X_{2})$, $(X_{1},X_{3})$, $(X_{2},X_{3})$, and
$(X_{1},X_{2},X_{3})$, calculated separately for the uncorrelated
case and the three correlated cases. Three key findings jump out
as follows. First, the total sensitivity indices
from the generalized ADD for all three correlated cases comprise both
variance- and covariance-driven contributions, whereas the total sensitivity
indices from the generalized ADD for the uncorrelated case or from
the classical ADD emanate solely from the variances of component functions.
Second, for the mixedly correlated case, the sum of the variance-driven
indices may exceed unity, while the sum of the covariance-driven indices
may be negative, as specifically demonstrated when $\rho_{12}=-1/5$,
$\rho_{13}=2/5$, $\rho_{23}=-4/5$. Third, the stronger the correlations
among random variables, the larger the covariance-driven contributions
to the total sensitivity indices.

Using the total sensitivity indices in Table \ref{table1}, the total
effect sensitivity indices of $y$ with respect to $X_{i}$, defined
as $\bar{S}_{i}:=\sum_{i\in u}S_{u}$, $i=1,2,3$, were calculated
to decipher the importance of each random variable. Table \ref{table3}
displays the total effect sensitivity indices with respect to three
random variables for the four cases of correlation properties. The
parenthetical numbers indicate relative rankings of all three variables,
except when there is a tie. For identical correlation structures,
such as the uncorrelated and equally correlated cases, all three variables
are equally important, yielding a three-way tie, as $y$ is a symmetric
function. For the positively correlated case, where the correlation
coefficient increases monotonically from $1/5$ to $4/5$, $X_{1}$
and $X_{3}$ are the least and the most important variables, respectively,
while the significance of $X_{2}$ is intermediary. The order of ranking
should reverse if the correlation coefficient decreases monotonically.
When the correlation coefficients are both positive and negative,
that is, for the mixedly correlated case, $X_{1}$ and $X_{2}$ become
the most and the least important variables, respectively. Clearly,
the correlation structure of random variables heavily influences the
composition of component functions as well as global sensitivity analysis.

\begin{table}
\caption{Triplets of global sensitivity indices from the generalized ADD of
$y$$^{(\text{a})}$ }

\begin{centering}
\begin{tabular}{c|ccc|ccc}
\hline
 & {\footnotesize $S_{u,v}$ } & {\footnotesize $S_{u,c}$ } & {\footnotesize $S_{u}$ } & {\footnotesize $S_{u,v}$ } & {\footnotesize $S_{u,c}$ } & {\footnotesize $S_{u}$ }\tabularnewline
\cline{2-7}
 & \multicolumn{3}{c|}{{\footnotesize Case 1: Uncorrelated }} & \multicolumn{3}{c}{{\footnotesize Case 2: Equally correlated}}\tabularnewline
{\footnotesize $\mathbf{X}_{u}$ } & \multicolumn{3}{c|}{{\footnotesize $(\rho_{12}=\rho_{13}=\rho_{23}=0)$ }} & \multicolumn{3}{c}{{\footnotesize $(\rho_{12}=\rho_{13}=\rho_{23}=1/5)$ }}\tabularnewline
\hline
{\footnotesize $X_{1}$ } & {\footnotesize 0.313725} & {\footnotesize 0} & {\footnotesize 0.313725} & {\footnotesize 0.227088} & {\footnotesize 0.089780} & {\footnotesize 0.316868}\tabularnewline
{\footnotesize $X_{2}$ } & {\footnotesize 0.313725} & {\footnotesize 0} & {\footnotesize 0.313725} & {\footnotesize 0.227088} & {\footnotesize 0.089780} & {\footnotesize 0.316868}\tabularnewline
{\footnotesize $X_{3}$ } & {\footnotesize 0.313725} & {\footnotesize 0} & {\footnotesize 0.313725} & {\footnotesize 0.227088} & {\footnotesize 0.089780} & {\footnotesize 0.316868}\tabularnewline
{\footnotesize $(X_{1},X_{2})$ } & {\footnotesize 0.019608} & {\footnotesize 0} & {\footnotesize 0.019608} & {\footnotesize 0.012349} & {\footnotesize 0.004116} & {\footnotesize 0.016465}\tabularnewline
{\footnotesize $(X_{1},X_{3})$ } & {\footnotesize 0.019608} & {\footnotesize 0} & {\footnotesize 0.019608} & {\footnotesize 0.012349} & {\footnotesize 0.004116} & {\footnotesize 0.016465}\tabularnewline
{\footnotesize $(X_{2},X_{3})$ } & {\footnotesize 0.019608} & {\footnotesize 0} & {\footnotesize 0.019608} & {\footnotesize 0.012349} & {\footnotesize 0.004116} & {\footnotesize 0.016465}\tabularnewline
{\footnotesize $(X_{1},X_{2},X_{3})$ } & {\footnotesize 0} & {\footnotesize 0} & {\footnotesize 0} & {\footnotesize 0} & {\footnotesize 0} & {\footnotesize 0}\tabularnewline
\hline \hline
{\footnotesize $\sum$ } & {\footnotesize 1} & {\footnotesize 0} & {\footnotesize 1} & {\footnotesize 0.718312} & {\footnotesize 0.281688} & {\footnotesize 1}\tabularnewline
\hline
 & \multicolumn{3}{c|}{{\footnotesize Case 3: Positively correlated}} & \multicolumn{3}{c}{{\footnotesize Case 4: Mixedly correlated}}\tabularnewline
 & \multicolumn{3}{c|}{{\footnotesize $(\rho_{12}=1/5,\rho_{13}=2/5,\rho_{23}=4/5)$}} & \multicolumn{3}{c}{{\footnotesize $(\rho_{12}=-1/5,\rho_{13}=2/5,\rho_{23}=-4/5)$}}\tabularnewline
\hline
{\footnotesize $X_{1}$ } & {\footnotesize 0.164847} & {\footnotesize 0.096992} & {\footnotesize 0.261839} & {\footnotesize 0.518299} & {\footnotesize 0.103039} & {\footnotesize 0.621337}\tabularnewline
{\footnotesize $X_{2}$ } & {\footnotesize 0.168309} & {\footnotesize 0.165600} & {\footnotesize 0.333909} & {\footnotesize 0.546677} & {\footnotesize -0.509771} & {\footnotesize 0.036905}\tabularnewline
{\footnotesize $X_{3}$ } & {\footnotesize 0.172897} & {\footnotesize 0.202314} & {\footnotesize 0.375211} & {\footnotesize 0.518116} & {\footnotesize -0.201389} & {\footnotesize 0.316728}\tabularnewline
{\footnotesize $(X_{1},X_{2})$ } & {\footnotesize 0.008812} & {\footnotesize 0.008812} & {\footnotesize 0.017624} & {\footnotesize 0.028623} & {\footnotesize -0.014311} & {\footnotesize 0.014311}\tabularnewline
{\footnotesize $(X_{1},X_{3})$ } & {\footnotesize 0.006049} & {\footnotesize 0.004321} & {\footnotesize 0.010370} & {\footnotesize 0.019647} & {\footnotesize -0.014034} & {\footnotesize 0.005614}\tabularnewline
{\footnotesize $(X_{2},X_{3})$ } & {\footnotesize 0.000786} & {\footnotesize 0.000262} & {\footnotesize 0.001048} & {\footnotesize 0.002553} & {\footnotesize 0.002553} & {\footnotesize 0.005105}\tabularnewline
{\footnotesize $(X_{1},X_{2},X_{3})$ } & {\footnotesize 0} & {\footnotesize 0} & {\footnotesize 0} & {\footnotesize 0} & {\footnotesize 0} & {\footnotesize 0}\tabularnewline
\hline \hline
{\footnotesize $\sum$ } & {\footnotesize 0.5217} & {\footnotesize 0.4783} & {\footnotesize 1} & {\footnotesize 1.63391} & {\footnotesize -0.63391} & {\footnotesize 1}\tabularnewline
\hline
\end{tabular}
\par\end{centering}

{\footnotesize ~~~(a) $y=12+4X_{1}+4X_{2}+4X_{3}+X_{1}X_{2}+X_{1}X_{3}+X_{2}X_{3}$,
where $a_{0}=b_{0}=c_{0}=2$, }{\footnotesize \par}

{\footnotesize ~~~~~~~~$a_{1}=b_{1}=c_{1}=1$.}{\footnotesize \par}

\label{table2}
\end{table}

\begin{table}
\caption{Total effects of random variables on the variance of $y$ and relative
rankings$^{(\text{a})}$ }

\begin{centering}
\begin{tabular}{c|c|c|c}
\hline
 & {\footnotesize $\bar{S}{}_{1}$ } & {\footnotesize $\bar{S}_{2}$ } & {\footnotesize $\bar{S}{}_{3}$ }\tabularnewline
{\footnotesize Case} & {\footnotesize (rank)} & {\footnotesize (rank)} & {\footnotesize (rank)}\tabularnewline
\hline
{\footnotesize Case 1: Uncorrelated } & {\footnotesize 0.352941} & {\footnotesize 0.352941} & {\footnotesize 0.352941}\tabularnewline
{\footnotesize $(\rho_{12}=\rho_{13}=\rho_{23}=0)$ } & {\footnotesize (Three-way tie)} & {\footnotesize (Three-way tie)} & {\footnotesize (Three-way tie)}\tabularnewline
\hline
{\footnotesize Case 2: Equally correlated} & {\footnotesize 0.349798} & {\footnotesize 0.349798} & {\footnotesize 0.349798}\tabularnewline
{\footnotesize $(\rho_{12}=\rho_{13}=\rho_{23}=1/5)$ } & {\footnotesize (Three-way tie)} & {\footnotesize (Three-way tie)} & {\footnotesize (Three-way tie)}\tabularnewline
\hline
{\footnotesize Case 3: Positively correlated} & {\footnotesize 0.289833} & {\footnotesize 0.352581} & {\footnotesize 0.386628}\tabularnewline
{\footnotesize $(\rho_{12}=1/5,\rho_{13}=2/5,\rho_{23}=4/5)$} & {\footnotesize (3)} & {\footnotesize (2)} & {\footnotesize (1)}\tabularnewline
\hline
{\footnotesize Case 4: Mixedly correlated} & {\footnotesize 0.641262} & {\footnotesize 0.056322} & {\footnotesize 0.327446}\tabularnewline
{\footnotesize $(\rho_{12}=-1/5,\rho_{13}=2/5,\rho_{23}=-4/5)$} & {\footnotesize (1)} & {\footnotesize (3)} & {\footnotesize (2)}\tabularnewline
\hline
\end{tabular}
\par\end{centering}

{\footnotesize ~~(a) The total effect of random variable $X_{i}$
is defined as $\bar{S}_{i}:=\sum_{i\in u}S_{u}$, $i=1,2,3$, $\emptyset\ne u\subseteq\{1,2,3\}$.}{\footnotesize \par}

\label{table3}
\end{table}

\subsection{A Random Eigenvalue Problem}
The next example is motivated on solving a practical problem, involving uncertainty quantification of natural frequencies of a vibrating cantilever plate, as shown in Figure \ref{figure1}(a).  The plate has the following deterministic geometric and material properties:  length $L=2$ in (50.8 mm), width $W=1$ in (25.4 mm), Young's modulus, $E=30\times 10^6$ psi (206.8 GPa), Poisson's ratio $\nu=0.3$, and mass density $\rho=7.324\times 10^{-4}$ lb-s$^2$/in$^4$ (7827 kg/mm$^3$). The randomness in natural frequencies arises due to random thickness $t(\xi)$, which is spatially varying in the longitudinal direction $\xi$ only.  The thickness is represented by a homogeneous, lognormal random field $t(\xi)=c\exp[\alpha(\xi)]$ with mean $\mu_t=0.01$ in (0.254 mm), variance $\sigma_t^2=v_t^2 \mu_t^2$, and coefficient of variation $v_t=0.2$, where $c=\mu_t/\sqrt{1+v_t^2}$ and $\alpha(\xi)$ is a zero-mean, homogeneous, Gaussian random field with variance $\sigma_\alpha^2=\ln(1+v_t^2)$ and covariance function $\Gamma_\alpha(\tau)=\mathbb{E}[\alpha(\xi)\alpha(\xi+\tau)=\sigma_\alpha^2 \exp[-|\tau|/(0.2L)]$.  A $10 \times 20$ finite-element mesh of the plate, consisting of 200 eight-noded, second-order shell elements and 661 nodes, is shown in Figure \ref{figure1}(b).  Using this mesh and the well-known midpoint method, the random field $\alpha(\xi)$ was discretized into a zero-mean, 20-dimensional, dependent Gaussian random vector $\mathbf{X}$ with covariance matrix  $\mathbf{\Sigma}_{\mathbf{X}}=[\Gamma_\alpha(\xi_i-\xi_j)]$, $i,j=1,\cdots,20$, where $\xi_i$ is the coordinate of the center of the $i$th column of elements in Figure \ref{figure1}(b).  The same mesh was used to calculate the natural frequencies, which are square-root of eigenvalues.

\begin{figure}[h]
\begin{centering}
\includegraphics[scale=0.6]{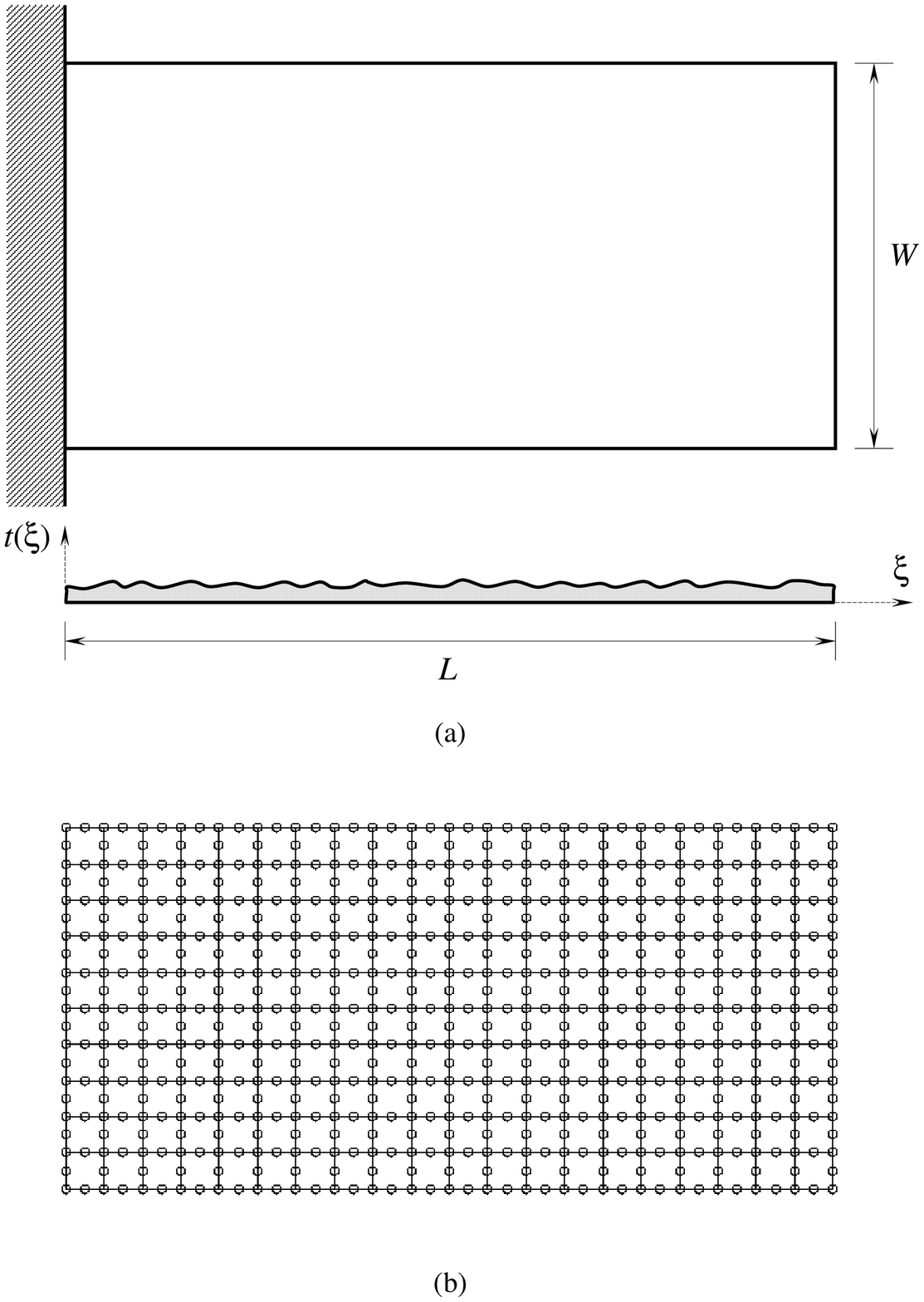}
\par\end{centering}
\caption{A cantilever plate; (a) geometry; (b) finite-element discrete model}
\label{figure1}
\end{figure}

The bivariate, second-order approximation and the bivariate, fourth-order approximation, that is, (\ref{r7}) truncated at $S=2$, $m=2$ and $S=2$, $m=4$, respectively, were employed to estimate various probabilistic characteristics of the first four natural frequencies of the plate. The construction of orthonormal polynomials is identical to that in the first example. However, the integrals involved in determining the coefficients of (\ref{r7}) were estimated from the dimension-reduction integration scheme \cite{xu04}, entailing at most two-dimensional integrations.  For the three-point ($=m+1$, $m=2$) and five-point ($=m+1$, $m=4$) Gauss-Hermite quadrature rules selected, the two proposed approximations require $20\times(20-1)(3-1)^2/2+(20\times(3-1)+1=801$ and $20\times(20-1)(5-1)^2/2+(20\times(5-1)+1=3121$
finite-element analyses, respectively \cite{xu04}.

Table \ref{table4} presents the means and standard deviations of the first four natural frequencies, $\omega_i$, $i=1,\cdots,6$, of the plate by three different methods: the two proposed bivariate approximations and crude Monte Carlo simulation (MCS).  In all three methods, the solution of the matrix characteristic equation for a given input is equivalent to performing a finite-element analysis.  Therefore, computational efficiency, even for this simple plate model, is a practical requirement in solving random eigenvalue problems.  The statistics by the proposed methods were obtained using 5000 samples of (\ref{r7}), which consist of repeated yet inexpensive evaluations of elementary functions. Due to expensive finite-element analysis, however, crude MCS was conducted only up to 5000 realizations, which should be adequate for providing benchmark solutions of the second-moment characteristics. The agreement between the means and standard deviations by the proposed methods and crude MCS in Table \ref{table4} is very good even for the second-order approximation.

\begin{table}
\caption{Means and standard deviations of the first four natural frequencies of the cantilever plate}
\begin{centering}
\begin{tabular}{cccccccccc}
\hline
 &  & \multicolumn{2}{c}{{\footnotesize $\begin{array}{c}
\mathrm{Generalized\: ADD}\\
(S=2,\: m=2)
\end{array}$}} &  & \multicolumn{2}{c}{{\footnotesize $\begin{array}{c}
\mathrm{Generalized\: ADD}\\
(S=2,\: m=4)
\end{array}$}} &  & \multicolumn{2}{c}{{\footnotesize $\begin{array}{c}
\mathrm{Crude\: MCS}\\
(5000\:\mathrm{samples})
\end{array}$}}\tabularnewline
\cline{3-4} \cline{6-7} \cline{9-10}
{\footnotesize Mode} &  & {\footnotesize $\begin{array}{c}
\mathrm{Mean}\\
(\mathrm{Hz})
\end{array}$} & {\footnotesize $\begin{array}{c}
\mathrm{St.\, dev.}\\
(\mathrm{Hz})
\end{array}$} &  & {\footnotesize $\begin{array}{c}
\mathrm{Mean}\\
(\mathrm{Hz})
\end{array}$} & {\footnotesize $\begin{array}{c}
\mathrm{St.\, dev.}\\
(\mathrm{Hz})
\end{array}$} &  & {\footnotesize $\begin{array}{c}
\mathrm{Mean}\\
(\mathrm{Hz})
\end{array}$} & {\footnotesize $\begin{array}{c}
\mathrm{St.\, dev.}\\
(\mathrm{Hz})
\end{array}$}\tabularnewline
\hline
{\footnotesize 1} &  & {\footnotesize 80.94} & {\footnotesize 15.68} &  & {\footnotesize 80.98} & {\footnotesize 15.97} &  & {\footnotesize 80.75} & {\footnotesize 17.06}\tabularnewline
{\footnotesize 2} &  & {\footnotesize 355.13} & {\footnotesize 54.21} &  & {\footnotesize 355.16} & {\footnotesize 55.22} &  & {\footnotesize 355.45} & {\footnotesize 55.31}\tabularnewline
{\footnotesize 3} &  & {\footnotesize 508.27} & {\footnotesize 64.21} &  & {\footnotesize 508.43} & {\footnotesize 64.98} &  & {\footnotesize 508.20} & {\footnotesize 68.58}\tabularnewline
{\footnotesize 4} &  & {\footnotesize 1169.63} & {\footnotesize 137.45} &  & {\footnotesize 1169.83} & {\footnotesize 140.64} &  & {\footnotesize 1170.55} & {\footnotesize 142.29}\tabularnewline
\hline
\end{tabular}
\par\end{centering}
\label{table4}
\end{table}

Figure \ref{figure2} depicts the marginal probability densities of the four natural frequencies by the proposed approximations and crude MCS.  Due to the computational expense inherent to finite-element analysis, the same 5000 samples generated for verifying the statistics in Table \ref{table4} were utilized to develop the histograms in Figure \ref{figure2}.  However, since the proposed methods yield explicit eigenvalue approximations, an arbitrarily large sample size, e.g., 50,000 in this particular example, was selected to sample (\ref{r7}) for estimating the respective densities.  Again, the results of the proposed methods and crude MCS match well, given the relatively small sample size of crude MCS.  Nonetheless, there exist slight discrepancies in the tail regions of a few densities, suggesting a need for improvements by invoking higher-variate approximations.

\begin{figure}[h]
\begin{centering}
\includegraphics[scale=0.63]{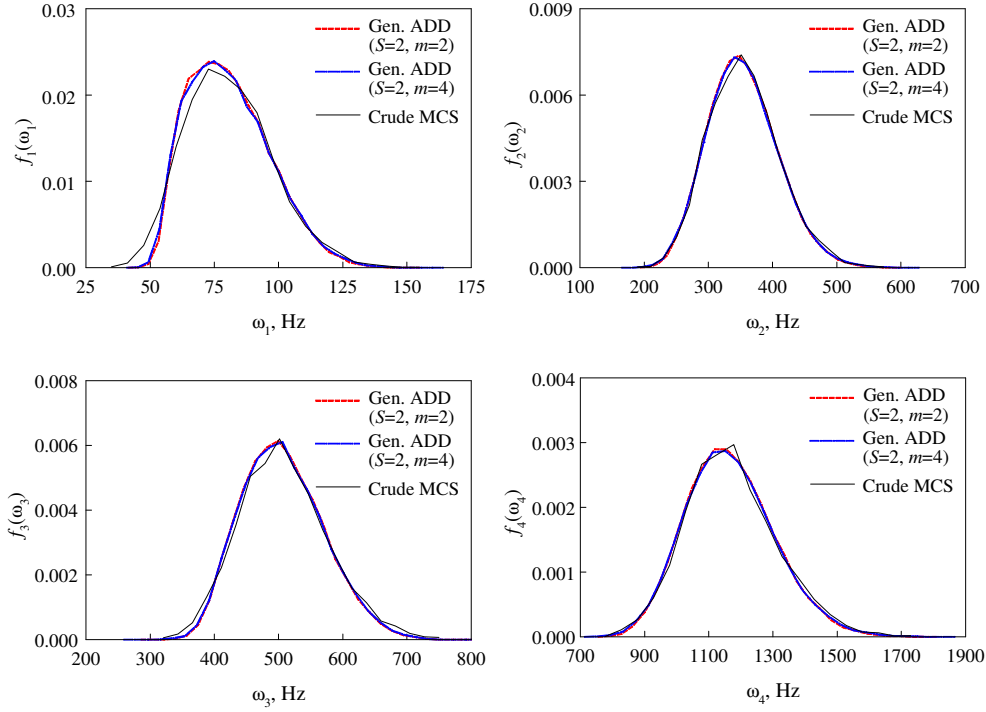}
\par\end{centering}
\caption{Marginal probability density functions of the first four natural frequencies of the cantilever plate}
\label{figure2}
\end{figure}

The proposed methods, especially the bivariate, second-order approximation, are computationally more efficient than crude MCS.  Comparing the results of Table \ref{table4}, the fourth-order approximation produces at most a modest improvement in the second-moment properties by the second-order approximation. Moreover, the respective marginal densities obtained by both approximations, see Figure \ref{figure2}, are practically coincident. Therefore, the second-order approximation provides satisfactory results without incurring the significantly higher cost of the fourth-order approximation, at least, for this example.  Having said so, the cost scaling of a bivariate approximation, whether second-order or fourth-order, is still quadratic with respect to the number of random variables.  Therefore, future efforts on developing adaptive-sparse approximations, where global sensitivity indices can be used to filter out unimportant component functions, should be explored.

\section{Conclusion}

A generalized ADD for dependent random variables, representing a finite
sum of lower-dimensional component functions of a multivariate function,
was studied. The classical annihilating conditions, when appropriately
weakened, reveal two important properties of the generalized ADD:
the component functions have \emph{zero} means and are hierarchically
orthogonal. A simple, alternative approach is proposed for deriving
the coupled system of equations satisfied by the component functions.
The coupled equations, which subsume as a special case the classical
ADD, reproduces the component functions for independent probability
measures. By exploiting measure-consistent, multivariate orthogonal
polynomials as bases, a new constructive method is proposed for determining
the component functions of the generalized ADD. The method leads to
a coupled system of linear algebraic equations for the expansion coefficients
of the component functions that is not only easy to implement and
solve, but also supports scalability to higher dimensions. New generalized
formulae are presented for the second-moment characteristics of a
general stochastic function, including three distinct global sensitivity
indices, relevant to dependent probability distributions. Analogous
to the component functions, the generalized formulae shrink to the
existing formulae from the classical ADD when the random variables
are independent. Gaining insights from the generalized ADD, two generalized
effective dimensions, one in the superposition sense and the other
in the truncation sense, are defined. Numerical results from a simple
yet insightful example indicate that the statistical dependence among
random variables induces higher-order terms in the generalized ADD
that may not be present in the original function or in the classical
ADD. In addition, the component functions depend significantly on
the correlation coefficients of random variables. Consequently, the
global sensitivity indices may also vary widely, producing distinct
rankings of random variables. Finally, an application to solving random eigenvalue problems
demonstrates that the proposed approximation provides not only accurate, but also computationally efficient, estimates of the statistical moments and probability densities of natural frequencies.

\section*{{\small Acknowledgments}}

The author wishes to acknowledge financial support from the U.S.
National Science Foundation under Grant Nos. CMMI-0969044 and CMMI-1130147.  Thanks to the three anonymous reviewers and the associate editor for providing important comments, which led to an improved, final version of the paper. The author benefited from discussions with Professors Giles Hooker (Cornell University), Palle Jorgensen (The University of Iowa), and Yuan Xu (University of Oregon).

\end{document}